\documentclass[12pt]{article}

\usepackage[]{graphicx,float,latexsym,times}
\usepackage{amsfonts,amstext,amsmath,amssymb,amsthm}
\usepackage{caption2}

\usepackage{bm}
\usepackage{amssymb,mathtools} 

\long\def\symbolfootnote[#1]#2{\begingroup%
\def\thefootnote{\fnsymbol{footnote}}\footnote[#1]{#2}\endgroup}

\newtheorem{theorem}{Theorem}[section]
\newtheorem{lemma}{Lemma}[section]

\theoremstyle{definition} 
\newtheorem{definition}{Definition}[section]
\newtheorem{remark}{Remark}[section]

\numberwithin{equation}{section} 

\usepackage[authoryear]{natbib}

\begin{document}

\title{Sequential Detection of Three-Dimensional Signals under Dependent Noise}

\date{}

\author{Annabel Prause and Ansgar Steland \\ \ \\  Institute of Statistics \\ RWTH Aachen University
}

\maketitle

\symbolfootnote[0]{\normalsize Address correspondence to Annabel Prause or Ansgar Steland,
Institute of Statistics, RWTH Aachen University, W\"ullnerstra{\ss}e 3, 52062 Aachen, Germany; E-mail: prause@stochastik.rwth-aachen.de or steland@stochastik.rwth-aachen.de}
\symbolfootnote[0]{\normalsize This is a preprint version of an article to appear in {\em Sequential Analysis}.}

\begin{abstract}
{\small \noindent\textbf{Abstract:} We study detection methods for multivariable signals under dependent noise. The main focus is  on three-dimensional signals, i.e.\ on signals in the space-time domain. Examples for such signals are multifaceted. They include geographic and climatic data as well as image data, that are observed over a fixed time horizon.
We assume that the signal is observed as a finite block of noisy samples whereby we are interested in detecting changes from a given reference signal.
Our detector statistic is based on a sequential partial sum process, related to classical signal decomposition and reconstruction approaches applied to the sampled signal. 
We show that this detector process converges weakly under the no change null hypothesis that the signal coincides with the reference signal, provided that the spatial-temporal partial sum process associated to the random field of the noise terms disturbing the sampled signal converges to a Brownian motion. 
More generally, we also establish the limiting distribution under a wide class of local alternatives that allows for smooth as well as discontinuous changes. 
Our results also cover extensions to the case that the reference signal is unknown.
We conclude with an extensive simulation study of the detection algorithm.} 
\end{abstract}


\section{Introduction} 

Signal processing and signal transmission play an important role in many different areas. Typical problems include the reconstruction of a signal by its discretely sampled values as well as the detection of changes from a given reference signal. 
For univariate signals $f\colon\mathbb{R}\to\mathbb{R}$, sampled equidistantly using a sampling period $\tau>0$ and disturbed by additive noise, such that one obtains a block of noisy samples $\{y_{\bm i}: i\in\{1,\ldots,n\}\}$ where
\begin{align}\label{basicmodelpreface}
y_{\bm i}=f\left(i\tau\right)+\varepsilon_{\bm i},
\end{align}
a nonparametric joint reconstruction/detection algorithm has been proposed in the paper of \cite{PawlakSteland}.
Their approach has several appealing features. Firstly, the algorithm can detect changes while reconstructing the signal at the same time. Secondly, it is a nonparametric approach, i.e.\ no further information about the exact class to which the observed signal belongs is necessary. Lastly, the procedure works in a sequential way such that changes can be detected on-line, in contrast to off-line detection schemes which can first detect changes in retrospect, i.e.\ when the whole data set is already available.

A natural question arises whether this approach also works for high-dimensional signals. 
One answer to this problem is given in \cite{PrauseSteland}, where the authors treat matrix-valued signals and apply results from \cite{PawlakSteland} by considering quadratic forms.
In the present paper, however, we consider a more general framework by focusing our attention on signals $f\colon\mathbb{R}^q\to\mathbb{R}$ for $q>1$. Examples for such signals are multifaceted, including geographic and climatic data as well as image data, that are observed over a fixed time horizon. 
In order to simplify the notation we fix $q=3$ as this case also covers the most interesting applications. However, our results also hold true for $q=2$ and arbitrary $q>3$ and the corresponding proofs can easily be completed along the same lines. Thus, in the following 
we are interested in reconstructing three-dimensional signals and, at the same time, in detecting changes from a given reference signal. Here, one component represents the time and the other two the location. The application that we have in mind are video signals, i.e.\ sequences of image frames over time.

The basis on which we now want to establish our investigations is a finite block of noisy samples $\{y_{\bm{i}}=y_{i_1,i_2,i_3}: (i_1,i_2,i_3)\in\{1,\ldots,n_1\}\times\{1,\ldots,n_2\}\times\{1,\ldots,n_3\}\}$ that -- in accordance with model (\ref{basicmodelpreface}) -- is obtained from the model
\begin{align}\label{modelmvb}
y_{i_1,i_2,i_3}=f\left(i_1\tau_1,i_2\tau_2,i_3\tau_3\right)+\varepsilon_{i_1,i_2,i_3}\,.
\end{align}
Here, $f(t,r_2,r_3)$ is the unknown signal depending on time ($i_1$) and location ($i_2$ and $i_3$), $\left\{\varepsilon_{\bm{i}}=\varepsilon_{i_1,i_2,i_3}\right\}$ is a zero mean random field and $\tau_j=\tau_{n_j},\,j=1,2,3$, are the sampling periods. We assume that they fulfill $\tau_j\to 0$ and $n_j\tau_j\to\overline{\tau}_j$, $j=1,2,3$ as $\min_{1\leq i\leq 3}n_i\to\infty$.
As in \cite{PawlakSteland} we want to base our approaches on classical reconstruction procedures from the signal sampling theory, leading to sequential partial sum processes as detector statistics, see Section \ref{DetAlg}. In order to make these detector statistics applicable we need to determine proper control limits; thus, in Section \ref{LimDist} we will show that we can generalize the two main weak convergence results in \cite{PawlakSteland} to our multidimensional context, i.e.\ we show weak convergence of the detection process towards Gaussian processes under different assumptions on the dependence structure of the noise processes where either the null hypothesis $f=f_0$ or the alternative $f\neq f_0$ holds true. In Section \ref{Ext} we present extensions to weighting functions, which allow to detect the location of the change as well, and discuss how to treat the case of an unknown but time-constant reference signal. 
Finally, in Section \ref{sim} we present some simulation results concerning the rejection rates and the power of the detection algorithm.

\section{The detection algorithm}\label{DetAlg}
We now want to extend the main results of \cite{PawlakSteland} to signals $f\colon\mathbb{R}^q\to\mathbb{R}$ with $q=3$. 
%
As in \cite{PawlakSteland} we base our estimator of $f(t,r_2,r_3)$ on results of the signal sampling theory like the Shannon-Whittaker theorem. This theorem has generalizations to signals with several variables. In three dimensions we have for band-limited functions on $[-\Omega_1,\Omega_1]\times[-\Omega_2,\Omega_2]\times[-\Omega_3,\Omega_3]$ with $0<\Omega_1,\Omega_2,\Omega_3<\infty$ the representation
\begin{align*}
f(t,r_2,r_3)=&\sum_{i_1=-\infty}^{\infty}\sum_{i_2=-\infty}^{\infty}\sum_{i_3=-\infty}^{\infty}f\left(\frac{i_1\pi}{\Omega_1},\frac{i_2\pi}{\Omega_2},\frac{i_3\pi}{\Omega_3}\right)\\&\ \ \textnormal{sinc}\left(\Omega_1\left(t-\frac{i_1\pi}{\Omega_1}\right)\right)\textnormal{sinc}\left(\Omega_2\left(r_2-\frac{i_2\pi}{\Omega_2}\right)\right)\textnormal{sinc}\left(\Omega_3\left(r_3-\frac{i_3\pi}{\Omega_3}\right)\right),
\end{align*}
where $\textnormal{sinc}(x)=\frac{\sin(x)}{x}$, cf.\ \cite{Jerri}, p.~1571. 
The most direct idea to construct an estimator of $f(t,r_2,r_3)$ would now be to just replace the values of $f$ at $\left(i_1\pi/\Omega_1,i_2\pi/\Omega_2,i_3\pi/\Omega_3\right)$ by the noisy observations $y_{\bm{i}}=y_{i_1,i_2,i_3}$. However, \cite{PawlakStadtmueller} have shown that this naive estimator is not even consistent in one dimension. Instead, they propose a post-filtering correction of the so-called oversampled version of the Shannon-Whittaker series to filter out high frequencies. This is the approach that we also adapt here. In three dimensions this oversampled version is of the form
\begin{align*}
f(t,r_2,r_3)=&\sum_{i_1=-\infty}^{\infty}\sum_{i_2=-\infty}^{\infty}\sum_{i_3=-\infty}^{\infty}f\left(i_1\tau_1,i_2\tau_2,i_3\tau_3\right)\\&\ \ \textnormal{sinc}(\pi\tau_1^{-1}(t-i_1\tau_1))\textnormal{sinc}(\pi\tau_2^{-1}(r_2-i_2\tau_2))\textnormal{sinc}(\pi\tau_3^{-1}(r_3-i_3\tau_3))
\end{align*}
for $0<\tau_j\leq\pi/\Omega_j, j=1,2,3$. If we convolve this version with
\begin{align*}
g(t,r_2,r_3)=\frac{\Omega_1\Omega_2\Omega_3}{\pi^3}\textnormal{sinc}(\Omega_1 t)\textnormal{sinc}(\Omega_2 r_2)\textnormal{sinc}(\Omega_3 r_3),
\end{align*}
use the fact that
\begin{align*}
&\quad\ \textnormal{sinc}(\pi\tau_1^{-1}(t-i_1\tau_1))\textnormal{sinc}(\pi\tau_2^{-1}(r_2-i_2\tau_2))\textnormal{sinc}(\pi\tau_3^{-1}(r_3-i_3\tau_3))\ast g(t,r_2,r_3)\\&
=\tau_1\tau_2\tau_3\frac{\Omega_1\Omega_2\Omega_3}{\pi^3} \textnormal{sinc}(\Omega_1(t-i_1\tau_1))\textnormal{sinc}(\Omega_2(r_2-i_1\tau_2))\textnormal{sinc}(\Omega_3(r_3-i_1\tau_3)),
\end{align*}
and replace $f\left(i_1\tau_1,i_2\tau_2,i_3\tau_3\right)$ by the noisy sample $y_{i_1,i_2,i_3}$,
this finally leads to the following truncated convolution form as an estimator for $f(t,r_2,r_3),$ namely
\begin{align*}
\widehat{f}_{n_1,n_2,n_3}(t,r_2,r_3)\coloneqq\tau_1\tau_2\tau_3
\sum_{i_1=1}^{n_1}\sum_{i_2=1}^{n_2}\sum_{i_3=1}^{n_3}y_{i_1,i_2,i_3}\varphi\left(t-i_1\tau_1,r_2-i_2\tau_2,r_3-i_3\tau_3\right),
\end{align*}
cf.\ \cite{PawlakStadtmueller}, p.~1427, and \cite{PawlakStadtmueller2}, p.~2527.
Here,
\begin{align*}
\varphi(t,r_2,r_3)=\frac{\sin(\Omega_1 t)}{\pi t}\frac{\sin(\Omega_2 r_2)}{\pi r_2}\frac{\sin(\Omega_3 r_3)}{\pi r_3}
\eqqcolon\widetilde{\varphi}_1(t)\widetilde{\varphi}_2(r_2)\widetilde{\varphi}_3(r_3)
\end{align*}
is a three-dimensional product reconstruction kernel with $\widetilde{\varphi}_j(0)=\Omega_j/\pi$ for $j=1,2,3$.

Given a reference signal $f_0(t,r_2,r_3)$, our aim now is to decide whether or not we can reject the null hypothesis 
\[
H_0: f(t,r_2,r_3)=f_0(t,r_2,r_3)
\]
for all $t\in[0,\overline{\tau}_1], r_2\in[0,\overline{\tau}_2], r_3\in[0,\overline{\tau}_3]$. As we receive our data in a sequential way over time as a sequence of image frames, we want to be able to detect changes as early as possible, i.e.\ we want to give an alarm as soon as we have enough evidence in our samples 
\[
\left\{y_{i_1,i_2,i_3}: i_1\in\{1,\ldots,k\},i_2\in\{1,\ldots,n_2\},i_3\in\{1,\ldots,n_3\}\right\}, 
\]
corresponding to the first $k$ image frames, to reject the null hypothesis. To achieve this aim we consider a sequential partial sum process over time which is defined as
\begin{align}\label{detectionprocess}
\mathcal{F}_n(s,t,r_2,r_3)&\coloneqq(\tau_1\tau_2\tau_3)^{-1/2}(\widehat{f}_{\left\lfloor n_1s \right\rfloor}(t,r_2,r_3)-\mathbb{E}_0\widehat{f}_{\left\lfloor n_1s \right\rfloor}(t,r_2,r_3))\notag\\&
=\sqrt{\tau_1\tau_2\tau_3}
\sum_{l_1=1}^{\lfloor n_1s\rfloor}\sum_{l_2=1}^{n_2}\sum_{l_3=1}^{n_3}\left[y_{l_1,l_2,l_3}-f_0(l_1\tau_1,l_2\tau_2,l_3\tau_3)\right]\notag\\&
\hspace{4.3cm}\varphi\left(t-l_1\tau_1,r_2-l_2\tau_2,r_3-l_3\tau_3\right),
\end{align}
for $0<s_0\leq s\leq 1, t\in[0,\overline{\tau}_1], r_2\in[0,\overline{\tau}_2], r_3\in[0,\overline{\tau}_3]$. The subscript $n$ denotes the dependence on the sample sizes. Note that $\mathbb E_0$ denotes as usual the expectation taken under the null hypothesis.
With this process we can easily define detectors such as the local detector
\begin{align}\label{locmvb}
\mathcal L_n\coloneqq
	\min\left\{n_0\leq k\leq n_1: \max_{r_2\in[0,\overline{\tau}_2],\atop r_3\in[0,\overline{\tau}_3]}\left|\mathcal{F}_n\left(\frac{k}{n_1},\frac{\overline{\tau}_1k}{n_1},r_2,r_3\right)\right|>c_L\right\}
\end{align}
or the global maximum norm detector
\begin{align}\label{globmvb}
\mathcal M_n\coloneqq\min\left\{n_0\leq k\leq n_1: \max_{0\leq t\leq \overline{\tau}_1 k/n_1} \max_{r_2\in[0,\overline{\tau}_2],\atop r_3\in[0,\overline{\tau}_3]}\left|\mathcal{F}_n\left(\frac{k}{n_1},t,r_2,r_3\right)\right|>c_M\right\}
\end{align}
with control limits $c_L>0$ and $c_M>0$ and $n_0\coloneqq\lfloor n_1s_0\rfloor$, $s_0\in(0,1)$. The reason to start monitoring in $n_0$ is that we assume that we have a kind of learning sample, i.e.\ we assume
\begin{align}\label{learningsample}
f(t,r_2,r_3)=f_0(t,r_2,r_3),\ 0\leq t\leq s_0\overline{\tau}_1, 0<s_0<1.
\end{align}
This assumption guarantees that no initial change in the signal occurs before the monitoring procedure starts which allows to estimate unknown parameters of the detection process such as the asymptotic variance $\sigma^2$, see below. 

Now the question arises how to reasonably choose the control parameters $c_L$ and $c_M$. To answer this question we are interested in the limiting distribution of our detector statistics. These can be derived from the limiting distribution of our stochastic process $\mathcal{F}_n(s,t,r_2,r_3)$ which is subject of the next section.

\section{Process limit distributions}\label{LimDist}
The asymptotic results to be discussed now, under the null hypothesis of no change and a rich class of alternative hypotheses under which the true signal converges to the reference signal, are based on a weak assumption about the asymptotic distribution of the partial sum process of the random field $\left\{\varepsilon_{i_1,i_2,i_3}\right\}$. Before discussing that assumption and providing the asymptotic results, we introduce some preliminaries and notations used in the sequel.
\subsection{Preliminaries}
As usual, the cardinality of a set $u\subseteq\{1,\ldots, q\}$ is denoted as $|u|$. Moreover, for $v\subseteq\{1,\ldots, q\}$ we write $u-v$ for the complement of $v$ with respect to $u$. In particular, we just write $-v$ if we take the complement of $v$ with respect to the whole set $\{1,\ldots, q\}$. Sets of the form $\{j,j+1,\ldots,k\}$ for integers $j$ and $k$ with $j\leq k$ are abbreviated as $j:k$, such that $\{1,\ldots, q\}=1:q$.

To pick out the components of a vector $\bm{x}\in\mathbb R^q$ that correspond to a set $u\subseteq 1:q$, we write $\bm{x}_u$, i.e.\ $\bm{x}_u$ stands for a vector with $|u|$ components with selected entries of $\bm{x}$. Now, let $u,v\subseteq 1:q$ and $\bm{x},\bm z\in [\bm{a},\bm{b}]$ with $u \cap v=\emptyset$. The symbol $\bm{x}_u:\bm z_v$ then denotes the point $\bm y\in\left[\bm a_{u \cup v},\bm b_{u \cup v}\right]$, where $y_j=x_j$ for all $j\in u$ and $y_j=z_j$ for all $j\in v$. 

Recall the concept of the $q$-fold alternating sum of a function $f$ over the hyperrectangle $[\bm{a},\bm{b}]$\index{$q$-fold alternating sum} which is defined as
\begin{align}\label{alternierendeSumme}
\Delta(f;\bm a,\bm b)\coloneqq\sum_{v\subseteq\{1,\ldots,q\}}(-1)^{|v|}f\left(\bm a_v:\bm b_{-v}\right).
\end{align}

Now, let $W$ be a subset of $[0,1]^q$. For points in $[0,1]^q$ set $\bm{t}=(t_1,\ldots,t_q)$ and $\bm{s}=(s_1,\ldots,s_q)$ respectively. We call $W$ a \textit{block} if it is of the form
\[
W\coloneqq\prod_{j=1}^{q}(s_j,t_j],
\]
where each $(s_j,t_j], j=1,\ldots,q$ is a left-open, right-closed subintervals of $[0,1]$. We now define the increment $X(W)$ of a stochastic process $X=\{X(\bm{t}):\bm{t}\in[0,1]^q\}$ around a block $W$ by means of the alternating sum (\ref{alternierendeSumme}) as
\[
X(W)\coloneqq\Delta(X;\bm s,\bm t)=\sum_{v\subseteq\{1,\ldots,q\}}(-1)^{|v|}X\left(\bm s_v:\bm t_{-v}\right).
\]
We are now able to define the Brownian motion on $[0,1]^q$, cf.\ \cite{Deo}, p.~709, where $C_q$ stands for the space of all real-valued continuous functions on $[0,1]^q$.
  \begin{definition}
The Brownian motion $B=\{B(\bm{t}):\bm{t}\in[0,1]^q\}$ on $[0,1]^q$\index{Brownian motion,!with multidimensional time set} is characterized by
\begin{itemize}
	\item[(a)] $P(B\in C_q)=1$,
	\item[(b)] if $W_1,\ldots,W_k$ are pairwise disjoint blocks in $[0,1]^q$, then the increments \[B(W_1),\ldots,B(W_k)\] are independent normal random variables with means zero and variances \[\lambda(W_1),\ldots,\lambda(W_k),\] $\lambda$ being the $q$-dimensional Lebesgue measure on $[0,1]^q$.
\end{itemize}
  \end{definition}

We can now formulate the main assumption for the asymptotic theory of the detection process as follows, where $\Rightarrow$ -- as usual -- stands for weak convergence in an appropriately chosen function space, see Appendix \ref{details}.

\textbf{Assumption 1:} Let $\{\varepsilon_{i_1,i_2,i_3}\}$ be a weakly stationary random field with $\mathbb E(\varepsilon_{\bm{i}})=0$ which satisfies a functional central limit theorem, i.e.\
\begin{align}\label{Ass1}
Z_n(v_1,v_2,v_3)\coloneqq(n_1n_2n_3)^{-1/2}\sum_{i_1=1}^{\lfloor n_1v_1\rfloor}\sum_{i_2=1}^{\lfloor n_2v_2\rfloor}
\sum_{i_3=1}^{\lfloor n_3v_3\rfloor}\varepsilon_{i_1,i_2,i_3}\Rightarrow\sigma B&(v_1,v_2,v_3),\\& (v_1,v_2,v_3)\in[0,1]^3,\notag
\end{align}
in the Skorohod space $D[0,1]^3$ as $\min_{1\leq i\leq 3}n_i\to\infty$ for some constant $\sigma^2\in(0,\infty)$.

Here, the constant $\sigma^2$ equals the long-run variance of the random field $\{\varepsilon_{i_1,i_2,i_3}\}$, i.e.\ 
\[
\sigma^2\coloneqq\lim_{\min_{1\leq i\leq 3}n_i\to\infty}\mathbb V\textnormal{ar}\left((n_1n_2n_3)^{-1/2}\sum_{i_1=1}^{n_1}\sum_{i_2=1}^{n_2}
\sum_{i_3=1}^{n_3}\varepsilon_{i_1,i_2,i_3}\right)=\sum_{\bm{k}\in\mathbb Z^3}\mathbb E(\varepsilon_{\bm{0}}\varepsilon_{\bm{k}}).
\]
There exist several results in the literature about the weak invariance principle (\ref{Ass1}) under specific conditions on the random field. In particular,
in the i.i.d.\ case we get the functional central limit theorem under the sole assumptions that 
\[
\mathbb E(\xi_{\bm{0}})=0,\ \mathbb E\left(\xi_{\bm{0}}^2\right)<\infty\ \textnormal{and}\ \sigma^2>0,
\]
see Corollary 1 in \cite{Wichura}.
More generally, a functional central limit theorem for strictly stationary and $\varphi$-mixing random fields can be found in \cite{Deo}, cf. Theorem 1.
Further results on weak invariance principles for random fields include weakly stationary associated as well as weakly stationary and $\alpha$-mixing random fields, cf.\ \cite{BulinskiKaene}, p.~2906, and \cite{BerkesMorrow}, Theorem~1, respectively. The latter obtain a strong approximation of the partial sum field by a Brownian motion from which one can deduce a weak invariance principle quite directly.
Other results on functional central limit theorems for random fields include the ones of \cite{WangWoodroofe}, cf.\ Theorem~1.1, and \cite{MachkouriVonlyWu}, cf.\ Theorem~2. These authors consider random fields of the form $X_{\bm{i}}=g(\varepsilon_{\bm{i}-\bm{s}};\bm{s}\in\mathbb Z^q)$ where $g$ is a measurable function and the $\{\varepsilon_{\bm{j}}; \bm{j}\in\mathbb Z^q\}$ are i.i.d.\ random variables. \cite{MachkouriVonlyWu} introduce the notion of a $p$-stable random field and then obtain a weak invariance principle for the so-called smoothed partial sum process.

\subsection{Asymptotics under the Null Hypothesis}
With the help of Assumption~1 we are now in a position to formulate the following theorem stating the asymptotic behaviour of the process.
\begin{theorem}\label{FCLTNull}
Suppose the noise process $\{\varepsilon_{\bm{i}}=\varepsilon_{i_1,i_2,i_3}\}$ meets Assumption~1. Assume that the sampling periods fulfill $n_j\tau_j\to\overline{\tau}_j$ for $j=1,2,3,$ as $\min_{1\leq i\leq 3}n_i\to\infty$. Then, under the null hypothesis $H_0$, we have
\[
\mathcal{F}_n(s,t,r_2,r_3)\Rightarrow\mathcal{F}(s,t,r_2,r_3),
\]
as $\min_{1\leq i\leq 3}n_i\to\infty$ for $0<s_0\leq s\leq 1$, $t\in[0,\overline{\tau}_1]$, $r_2\in[0,\overline{\tau}_2]$ and $r_3\in[0,\overline{\tau}_3]$.

The limit stochastic process $\mathcal{F}(s,t,r_2,r_3)$ is of the form
\begin{align*}
\mathcal{F}(s,t,r_2,r_3)\coloneqq\sqrt{\overline{\tau}_1\overline{\tau}_2\overline{\tau}_3}\,\sigma
\int_0^s\int_0^1\int_0^1&
\!\varphi\left(t-\overline{\tau}_1z_1,r_2-\overline{\tau}_2z_2,r_3-\overline{\tau}_3z_3\right)\\&
\mathrm{d}B(z_1,z_2,z_3),
\end{align*}
where $B(z_1,z_2,z_3)$ is the standard Brownian motion on $[0,1]^3$.
\end{theorem}
The weak convergence takes place in a higher dimensional Skorohod space and the last integral is interpreted as multivariate Riemann-Stieltjes integral, see Appendix \ref{details} for more details.

The next lemma is a characterization of the correlation structure of the limit process $\mathcal{F}(s,t,r_2,r_3)$.
  \begin{lemma}\label{covstructure}
\begin{itemize}
	\item[(a)] The process $\mathcal{F}(s,t,r_2,r_3)$ is a nonstationary multivariable Gaussian process with
\[ 
\mathbb E(\mathcal{F}(s,t,r_2,r_3))=\bm{0} 
\] 
and covariance function
\begin{align}\label{covfct}
&\quad\ \mathbb C\textnormal{ov}\left(\mathcal{F}(s^{(1)},t^{(1)},r^{(1)}_2,r^{(1)}_3),\mathcal{F}(s^{(2)},t^{(2)},r^{(2)}_2,r^{(2)}_3)\right)\notag\\&
=\sigma^2\overline{\tau}_1\overline{\tau}_2\overline{\tau}_3\int_0^{\min\{s^{(1)},s^{(2)}\}}\int_0^1\int_0^1\!\varphi(t^{(1)}-\overline{\tau}_1z_1,r^{(1)}_2-\overline{\tau}_2z_2,r^{(1)}_3-\overline{\tau}_3z_3)\notag\\&\hspace{3cm}
\varphi(t^{(2)}-\overline{\tau}_1z_1,r^{(2)}_2-\overline{\tau}_2z_2,r^{(2)}_3-\overline{\tau}_3z_3)\,\mathrm{d}z_3\mathrm{d}z_2\mathrm{d}z_1
\end{align}
for $0<s_0\leq s^{(1)},s^{(2)}\leq 1$, $0\leq t^{(1)},t^{(2)}\leq\overline{\tau}_1$, $0\leq r^{(1)}_2,r^{(2)}_2\leq\overline{\tau}_2$, and $0\leq r^{(1)}_3,r^{(2)}_3\leq\overline{\tau}_3$.
\item[(b)] The process $\mathcal F(s,t,r_2,r_3)$ has continuous sample paths.
\end{itemize}
  \end{lemma}
Now that we have the limit distribution of $\mathcal F_n(s,t,r_2,r_3)$ under the null hypothesis at our disposal, we can easily derive central limit theorems for the local and global maximum norm detector defined in (\ref{locmvb}) and (\ref{globmvb}). 
 \begin{lemma}\label{cltdet}
Assume that condition (\ref{learningsample}) holds true.
Then, under the conditions of Theorem~\ref{FCLTNull} the detectors satisfy the following central limit theorems:
\begin{align*}
&\mathcal L_n/n_1\Rightarrow\mathcal{L}\coloneqq\inf\left\{s\in[s_0,1]: \sup_{r_2\in[0,\overline{\tau}_2],\atop r_3\in[0,\overline{\tau}_3]}\left|\mathcal{F}(s,s\overline{\tau}_1,r_2,r_3)\right|>c_L\right\},\\&
\mathcal M_n/n_1\Rightarrow\mathcal{M}\coloneqq\inf\left\{s\in[s_0,1]: \sup_{0\leq t\leq s\overline{\tau}_1}\sup_{r_2\in[0,\overline{\tau}_2],\atop r_3\in[0,\overline{\tau}_3]}\left|\mathcal{F}(s,t,r_2,r_3)\right|>c_M\right\},
\end{align*}
as $\min_{1\leq i\leq 3} n_i\to\infty$. 
 \end{lemma}

\subsection{Asymptotics under the Alternative}
We now investigate the behaviour of our statistic $\mathcal{F}_n(s,t,r_2,r_3)$ under a general class of alternatives $H_1:f_{n_1,n_2,n_3}(t,r_2,r_3)\neq f_0(t,r_2,r_3)$, i.e.\ in the situation when the observed signal and the reference signal differ. We assume that our observed data $\{y_{\bm{i}}=y_{i_1,i_2,i_3}: (i_1,i_2,i_3)\in\{1,\ldots,n_1\}\times\{1,\ldots,n_2\}\times\{1,\ldots,n_3\}\}$ obey the following model:
\begin{align}\label{modelalternative3}
y_{i_1,i_2,i_3}=f_{n_1,n_2,n_3}\left(i_1\tau_1,i_2\tau_2,i_3\tau_3\right)+\varepsilon_{i_1,i_2,i_3}
\end{align}
with the true signal $f_{n_1,n_2,n_3}(t,r_2,r_3)$ depending on the sample size $(n_1,n_2,n_3)$ and $f_{n_1,n_2,n_3}(t,r_2,r_3)\to f_0(t,r_2,r_3)$ as $\min_{1\leq i\leq 3}n_i\to\infty$. 
The process $\{\varepsilon_{\bm i}=\varepsilon_{i_1,i_2,i_3}\}$ is again the zero mean noise random field fulfilling Assumption~1.

It turns out that the process $\mathcal F_n(s,t,r_2,r_3)$ converges to a well-defined and non-degenerate limit process under general conditions on the variation of the difference $f_{n_1,n_2,n_3}(t,r_2,r_3)-f_0(t,r_2,r_3)$, similar as in \cite{PawlakSteland}. However, whereas in dimension $q=1$ the Vitali variation suffices, in higher dimensions one has to consider the variation in the sense of Hardy and Krause.

Let $a,b\in\mathbb{R}$ with $a\leq b$. A \textit{ladder} on $[a,b]$ is a set $\mathcal Y$ containing $a$ and finitely many, possibly zero, values from $(a,b)$, see \cite{Owen}, p.~2. The successor of an element $y\in\mathcal Y$ is denoted by $y^+$. For $(y,\infty)\cap\mathcal Y=\emptyset$ we set $y^+=b$ and otherwise $y^+$ is the smallest element of $(y,\infty)\cap\mathcal Y$. In particular, if we consider a classical partition of $[a,b]$, we have $\mathcal Y=\{y_0,y_1,\ldots,y_m\}$ with $a=y_0<y_1<\ldots<y_m$ such that $y_{k+1}$ is the successor of $y_k$ for all $k<m$ and it is $b$ for $k=m$.

If we now define $\mathbb Y$ as the set of all ladders on $[a,b]$, we can write the total variation of a univariate function $f$ defined on $[a,b]$ as
\[
V(f;a,b)\coloneqq\sup_{\mathcal Y\in\mathbb Y}\sum_{y\in\mathcal Y}\left|f\left(y^+\right)-f(y)\right|.
\]

In order to generalize the one-dimensional variation to the multidimensional case we need the concept of multidimensional ladders. We now consider a hyperrectangle $[\bm{a},\bm{b}]$ with $\bm a,\bm b\in\mathbb R^q$ and $\bm a\leq \bm b$. We define a ladder $\mathcal Y$ on $[\bm{a},\bm{b}]$ as $\mathcal Y\coloneqq\prod_{j=1}^q\mathcal Y_j$, where $\mathcal Y_j$ is a ladder on $\left[a_j,b_j\right]$ for $j=1,\ldots,q$. Similarly, we say that $\bm y^+$ is the successor of $\bm y$ if $y^+_j$ is the successor of $y_j$ for each $j=1,\ldots,q$. Finally, we set $\mathbb Y\coloneqq\prod_{j=1}^q\mathbb Y_j$ for the set of all ladders on $[\bm{a},\bm{b}]$, where $\mathbb Y_j$ denotes the set of all ladders on $\left[a_j,b_j\right]$ for $j=1,\ldots,q$.
With these alternating sums we are now able to define the variation of $f$ over $\mathcal Y$ as
\[
V_{\mathcal Y}(f)\coloneqq\sum_{\bm y\in\mathcal Y}\left|\Delta\left(f;\bm y,\bm y^+\right)\right|.
\]
This leads to the following definition, see \cite{Owen}, Definition~1 and 2.
\begin{definition}
The variation of $f$ on the hyperrectangle $[\bm{a},\bm{b}]$, in the sense of Vitali, is
\begin{align*}
V_{[\bm{a},\bm{b}]}(f)\coloneqq\sup_{\mathcal Y\in\mathbb Y}V_{\mathcal Y}(f).
\end{align*}
The variation of $f$ on the hyperrectangle $[\bm{a},\bm{b}]$, in the sense of Hardy and Krause, is
\begin{align}\label{VarHardyKrause}
V_{HK}(f)=V_{HK}(f;\bm a,\bm b)\coloneqq\sum_{\emptyset\neq u\subseteq 1:q}V_{\left[\bm a_u,\bm b_u\right]}f\left(\bm{x}_u:\bm b_{-u}\right).
\end{align}
\end{definition}
Likewise as in the one-dimensional case, we say that the function $f$ is of bounded variation in the sense of Vitali (and we write $f\in BV$ or $f\in BV[\bm{a},\bm{b}]$) if $V(f)<\infty$. Note that for $q=1$ the variation in the sense of Vitali corresponds with the common definition of the variation of a univariate function.
We say that the function $f$ is of bounded variation in the sense of Hardy and Krause (and we write $f\in BVHK$ or $f\in BVHK [\bm{a},\bm{b}]$) if $V_{HK}(f)<\infty$. Note that the summand for $u=1:q$ in the above sum equals $V(f)$, such that $V(f)<\infty$ if $V_{HK}(f)<\infty$. Moreover, it was shown by \cite{Young}, that $V_{HK}(f)<\infty$ if and only if $V_{[\bm{a},\bm{b}]}(f)<\infty$ and $V_{\left[\bm a_u,\bm b_u\right]}f\left(\bm{x}_u:\bm z_{-u}\right)<\infty$ for all $0<|u|<q$ and all $\bm z_{-u}\in\left[\bm a_{-u},\bm b_{-u}\right]$, which is the original definition of bounded variation of Hardy, see \cite{Hardy}. This means, that in the above definition $\bm b_{-u}$ could be replaced by an arbitrary fixed point of the hyperrectangle $\left[\bm a_{-u},\bm b_{-u}\right]$ for $\emptyset\neq u\subseteq 1:q$.

For more details on these two notions of multidimensional variation we refer the reader to \cite{Owen}.
 
\subsubsection{Deterministic Disturbances}
We first consider the case where $f_{n_1,n_2,n_3}(t,r_2,r_3)$ is deterministic and start with
the following illustrative example, where $n=n_1=n_2=n_3$.
Consider
\begin{align*}
f_n(t,r_2,r_3)=f_0(t,r_2,r_3)+\frac{\left(t-\theta \overline{\tau}_1\right)^{\gamma}\widetilde{\delta}(r_2,r_3)}{n^{\beta}}\bm{1}_{\left\{t\geq\theta \overline{\tau}_1\right\}}
\end{align*}
for some $\theta\in (0,1)$, $\gamma\geq 0$, $\beta>0$, and a non-zero `location' function $\widetilde{\delta}(r_2,r_3)$ defined on $[0,\overline{\tau}_2]\times[0,\overline{\tau}_3]$. This is a local alternative with a change point at time $t=\theta \overline{\tau}_1$. This means, that up to time $\theta \overline{\tau}_1$ the observed data obey $f_0$ and after this point in time they get disturbed by $\left(t-\theta \overline{\tau}_1\right)^{\gamma}\widetilde{\delta}(r_2,r_3)/n^{\beta}$. This disturbance depends on the function $\widetilde{\delta}(r_2,r_3)$ which assigns different weights at the locations $(r_2,r_3)$ changing with the time.
In the following we require a more general model for local alternatives, namely we consider
\begin{align}\label{local3}
f_{n_1,n_2,n_3}(t,r_2,r_3)-f_0(t,r_2,r_3)=\frac{\delta(t,r_2,r_3)}{n_1^{\bm_1}n_2^{\beta_2}n_3^{\beta_3}}
\end{align}
for some $\beta_1,\beta_2,\beta_3>0$ and a deterministic function $\delta(t,r_2,r_3)$.
We assume that $\delta$ meets the following assumption.

\textbf{Assumption 2:} Let $\delta(t,r_2,r_3)$ be a nonzero function defined on $[0,\overline{\tau}_1]\times[0,\overline{\tau}_2]\times[0,\overline{\tau}_3]$ which is\\[-4ex]
\begin{itemize}
	\item[(a)] continuous
	\item[(b)] of bounded variation in the sense of Hardy and Krause. 
\end{itemize}
Now we are able to describe the asymptotic behaviour of the statistic $\mathcal{F}_n(s,t,r_2,r_3)$ under local alternatives.
\begin{theorem}\label{detalt}
Assume the sampling model in (\ref{modelalternative3}) with the local alternative given in (\ref{local3}) where $\beta=3/2$. Let Assumption~1 hold and suppose that either Assumption~2 (a) and $n\tau_j\to\overline{\tau}_j$, $n\to\infty$, $j=1,2,3,$ or Assumption~2 (b) and $n\tau_j=\overline{\tau}_j$, $j=1,2,3,$ is satisfied. Then, we have
\[
\mathcal{F}_n(s,t,r_2,r_3)\Rightarrow\mathcal{F}^{\delta}(s,t,r_2,r_3),
\]
as $\min_{1\leq i\leq 3}n_i\to\infty$ for $0<s_0\leq s\leq 1$, $t\in[0,\overline{\tau}_1]$, $r_2\in[0,\overline{\tau}_2]$ and $r_3\in[0,\overline{\tau}_3]$.

The limit stochastic process $\mathcal F^{\delta}(s,t,r_2,r_3)$ is given by
\begin{align*}
\mathcal{F}^{\delta}(s,t,r_2,r_3)\coloneqq\mathcal{F}(s,t,r_2,r_3)+\frac{1}{\sqrt{\overline{\tau}_1\overline{\tau}_2\overline{\tau}_3}}
\int_0^{s\overline{\tau}_1}\int_0^{\overline{\tau}_2}\int_0^{\overline{\tau}_3}
\!&\varphi\left(t-z_1,r_2-z_2,r_3-z_3\right)\\&\delta(z_1,z_2,z_3)\,\mathrm{d}z_3\mathrm{d}z_2\mathrm{d}z_1.
\end{align*}
\end{theorem}

Similarly as above we obtain central limit theorems for our detectors, defined in (\ref{locmvb}) and (\ref{globmvb}), under the alternative.
 \begin{lemma}\label{cltdetalt}
Let the condition in (\ref{learningsample}) hold. Then, under the assumptions of Theorem \ref{detalt} we obtain the asymptotic distribution of the local and global maximum norm detector by replacing $\mathcal F(s,t,r_2,r_3)$ by $\mathcal F^{\delta}(s,t,r_2,r_3)$ in Corollary \ref{cltdet}.
 \end{lemma}

\subsubsection{Random Disturbances}
We now consider random disturbances, i.e.\ we require that our data obey the model
\begin{align}\label{modelalternative2}
y_{i_1,i_2,i_3}=f_n\left(i_1\tau_1,i_2\tau_2,i_3\tau_3;\omega\right)+\varepsilon_{i_1,i_2,i_3}
\end{align}
where now $f_n(t,r_2,r_3;\omega)\to f_0(t,r_2,r_3)$ a.s. as $n\to\infty$.
To be more precise we require that
\begin{align}\label{local2}
f_n(t,r_2,r_3;\omega)-f_0(t,r_2,r_3)=\frac{\Delta(t,r_2,r_3;\omega)}{n^{\beta}}
\end{align}
for $\beta>0$, and $\Delta(t,r_2,r_3;\omega)$ being a random function that is independent of the random field $\{\varepsilon_{i_1,i_2,i_3}\}$. Moreover, we assume that $\Delta(t,r_2,r_3;\omega)\neq 0$ a.s.
We require that $\Delta(t,r_2,r_3;\omega)$ meets the following assumption.

\textbf{Assumption 3:} Let $\Delta(t,r_2,r_3;\omega)$ be an a.s.\ nonzero random function defined on $[0,\overline{\tau}_1]\times[0,\overline{\tau}_2]\times[0,\overline{\tau}_3]\times\Omega$ that is independent of the random field $\{\varepsilon_{i_1,i_2,i_3}\}$ and whose sample paths are\\[-4ex]
\begin{itemize}
	\item[(a)] continuous a.s. 
	\item[(b)] of bounded variation in the sense of Hardy and Krause a.s. 
\end{itemize}

Then we can describe the asymptotic behaviour of the statistic $\mathcal{F}_n(s,t,r_2,r_3)$ under random local alternatives.
\begin{theorem}\label{randalt}
Assume the sampling model in (\ref{modelalternative2}) with the local alternative given in (\ref{local2}) where $\beta=3/2$. Let Assumption~1 hold and suppose that either Assumption~3 (a) and $n\tau_j\to\overline{\tau}_j$, $n\to\infty$, $j=1,2,3,$ or Assumption~3 (b) and $n\tau_j=\overline{\tau}_j$, $j=1,2,3,$ is satisfied.
Then, we have
\[
\mathcal{F}_n(s,t,r_2,r_3)\Rightarrow\mathcal{F}^{\Delta}(s,t,r_2,r_3),
\]
as $\min_{1\leq i\leq 3}n_i\to\infty$ for $0<s_0\leq s\leq 1$, $t\in[0,\overline{\tau}_1]$, $r_2\in[0,\overline{\tau}_2]$ and $r_3\in[0,\overline{\tau}_3]$.

The limit stochastic process $\mathcal F^{\Delta}(s,t,r_2,r_3)$ is given by
\begin{align*}
\mathcal{F}^{\Delta}(s,t,r_2,r_3)\coloneqq\mathcal{F}(s,t,r_2,r_3)+\frac{1}{\sqrt{\overline{\tau}_1\overline{\tau}_2\overline{\tau}_3}}
\int_0^{s\overline{\tau}_1}\int_0^{\overline{\tau}_2}\int_0^{\overline{\tau}_2}
\!&\varphi\left(t-z_1,r_2-z_2,r_3-z_3\right)\\&\Delta(z_1,z_2,z_3;\omega)\,\mathrm{d}z_3\mathrm{d}z_2\mathrm{d}z_1.
\end{align*}
\end{theorem}

\section{Extensions: Weighting functions and unknown reference signals}\label{Ext}
In this section we want to demonstrate that we can easily extend some of the results of the previous section into several directions with only slight modifications. We give two examples that shall show the great flexibility and applicability of our results. We point out, among other things, that the detector statistic in (\ref{detectionprocess}) can serve with small changes not only as a detector for changes in time, but also as a detector for the concrete location where a change takes place. Moreover, we can extend the result in Theorem~\ref{FCLTNull} to allow for unknown reference signals, by appropriate centering of the observations.

\subsection{Additional Weighting Functions}
We begin with a generalization of the detector statistic in (\ref{detectionprocess}) in order to be able to detect the position of a change. This can be achieved by adding a suitable weighting function $w$ for the different pixels of the image and leads to the sequential monitoring process
\begin{align*}
\mathcal{F}^w_n(s,t,r_2,r_3)=\sqrt{\tau_1\tau_2\tau_3}
\sum_{l_1=1}^{\lfloor n_1s\rfloor}\sum_{l_2=1}^{n_2}&\sum_{l_3=1}^{n_3}\left[y_{l_1,l_2,l_3}-f_0(l_1\tau_1,l_2\tau_2,l_3\tau_3)\right]\\&
\varphi\left(t-l_1\tau_1,r_2-l_2\tau_2,r_3-l_3\tau_3\right)w(l_2\tau_2,l_3\tau_3,r_2,r_3)
\end{align*}
for $0<s_0\leq s\leq 1, t\in[0,\overline{\tau}_1], r_2\in[0,\overline{\tau}_2], r_3\in[0,\overline{\tau}_3]$. Before we get more specific about possible forms of $w$, we first want to reformulate Theorem~\ref{FCLTNull} for the new detection process $\mathcal{F}^w_n$.
\begin{theorem}\label{basicresultmvbw2}
Let $w$ be a continuous function with
\[
\sup_{r_2\in[0,\overline{\tau}_2],r_3\in[0,\overline{\tau}_3]}V_{HK}(w(\cdot,\cdot,r_2,r_3))<\infty.
\]
Suppose the noise process $\{\varepsilon_{\bm i}=\varepsilon_{i_1,i_2,i_3}\}$ meets Assumption 1. We assume that the sampling periods fulfill $n_j\tau_j\to\overline{\tau}_j$, $j=1,2,3,$ as $\min_{1\leq i\leq 3}n_i\to\infty$. Then, under the null hypothesis $H_0$, we have
\[
\mathcal{F}^w_n(s,t,r_2,r_3)\Rightarrow\mathcal{F}^w(s,t,r_2,r_3),
\]
as $\min_{1\leq i\leq 3}n_i\to\infty$ for $0<s_0\leq s\leq 1$, $t\in[0,\overline{\tau}_1]$, $r_2\in[0,\overline{\tau}_2]$, and $r_3\in[0,\overline{\tau}_3]$.

The limit stochastic process $\mathcal{F}^w(s,t,r_2,r_3)$ is of the form
\begin{align*}
\mathcal{F}^w(s,t,r_2,r_3)\coloneqq\sqrt{\overline{\tau}_1\overline{\tau}_2\overline{\tau}_3}\,\sigma
\int_0^s\int_0^1\int_0^1
&\!\varphi\left(t-\overline{\tau}_1z_1,r_2-\overline{\tau}_2z_2,r_3-\overline{\tau}_3z_3\right)\\&w(\overline{\tau}_2z_2,\overline{\tau}_3z_3,r_2,r_3)
\,\mathrm{d}B(z_1,z_2,z_3),
\end{align*}
where $B(z_1,z_2,z_3)$ is the standard Brownian motion on $[0,1]^3$.
\end{theorem}
Now the question arises which forms would be suitable for $w$. To answer this question it is useful to know the approximate form of the change that one wants to detect. If the aim is, for example, to detect a simple rectangle with length $c_2>0$ and width $c_3>0$ that occurs at a certain point in time, one could define $w$ for $\delta>0$ as
\begin{align}\label{charfctw}
w(z_2,z_3,r_2,r_3)=
\begin{cases}
1,  & \text{if } |z_2-r_2|\leq c_2, |z_3-r_3|\leq c_3,\\
0,  & \text{if } |z_2-r_2|\geq c_2+\delta, |z_3-r_3|\geq c_3+\delta,\\
\textnormal{smooth}, & \text{otherwise}.
\end{cases}
\end{align}
Then, one could define suitable detectors as before, e.g.\ the local detector as
\[\mathcal L^w_n\coloneqq
	\min\left\{n_0\leq k\leq n_1: \max_{l\in\{0,\ldots,n_2\},\atop m\in\{0,\ldots,n_3\}}\left|\mathcal{F}^w_n\left(\frac{k}{n_1},\frac{\overline{\tau}_1k}{n_1},\frac{\overline{\tau}_2l}{n_2},\frac{\overline{\tau}_3m}{n_3}\right)\right|>c^w_L\right\}
\]
and the global maximum norm detector as
\[\mathcal M^w_n\coloneqq
	\min\left\{n_0\leq k\leq n_1: \max_{0\leq t\leq\frac{\overline{\tau}_1k}{n_1}}\max_{r_2\in[0,\overline{\tau}_2],\atop r_3\in[0,\overline{\tau}_3]}\left|\mathcal{F}^w_n\left(\frac{k}{n_1},t,r_2,r_3\right)\right|>c^w_M\right\}
\]
for control limits $c^w_L>0$ and $c^w_M>0$.
Again, the application of the continuous mapping theorem leads to central limit theorems for these detectors. If, for example, the local detector $\mathcal L_n$ exceeds its control limit $c^w_L$ for $(k^{\star},l^{\star},m^{\star})$, we know that the rectangle occured on the $\overline{\tau}_1k^{\star}/n_1$th image frame at position $\left(\overline{\tau}_2l^{\star}/n_2,\overline{\tau}_3m^{\star}/n_3\right)$.
\begin{remark}
It is important that one chooses the weight function $w$ not only as a characteristic function, as then the continuity assumption of Theorem~\ref{basicresultmvbw2} on $w$ is not fulfilled. Moreover, characteristic functions with a domain that is not parallel to the axes have infinite variation which encourages the `smoothing' of $w$ as well, cf.\ \cite{Owen}, p.~14.  If one also wants to allow for characteristic functions without smoothing one has to consider the integrals as It\^{o} integrals instead of considering them as Riemann Stieltjes integrals as it is done in this work.
\end{remark}
The detection of more complex forms for changes than rectangles is possible as well by choosing corresponding domains for the smoothed characteristic function in (\ref{charfctw}).

\subsection{The Case of an Unknown Reference Signal}
If the reference signal is unknown to us, the above procedures are not applicable. To overcome this drawback, we shall assume in the sequel that the reference signal is time-constant, i.e.\
\begin{align}\label{timeconstant}
f_0(t,r_2,r_3)=f_0(0,r_2,r_3),\quad (r_2,r_3)\in[0,\overline{\tau}_2]\times[0,\overline{\tau}_3]
\end{align}
holds true for all $t\in[0,\overline{\tau}_1]$. In this case one may center the spatial-temporal observations at appropriately defined averages of previous observations. Here one can either use the learning sample 
\[
\left\{y_{\bm i}: \bm i\in\left\{1,\ldots,n_0\right\}\times\left\{1,\ldots,n_2\right\}\times\left\{1,\ldots,n_3\right\}\right\}
\]
or include all observations available at the current time instant. For $(l_2,l_3)\in\left\{1,\ldots,n_2\right\}\times\left\{1,\ldots,n_3\right\}$ let us define
\[
\overline{y}_{\cdot,l_2,l_3}\coloneqq\frac{1}{\left\lfloor n_1s_0\right\rfloor}\sum_{l_1=1}^{\left\lfloor n_1s_0\right\rfloor}y_{l_1,l_2,l_3},
\]
and consider
\begin{align*}
\mathcal{F}^c_n(s,t,r_2,r_3)\coloneqq\sqrt{\tau_1\tau_2\tau_3}
\sum_{l_1=1}^{\lfloor n_1s\rfloor}\sum_{l_2=1}^{n_2}\sum_{l_3=1}^{n_3}&(y_{l_1,l_2,l_3}-\overline{y}_{\cdot,l_2,l_3})
\varphi\left(t-l_1\tau_1,r_2-l_2\tau_2,r_3-l_3\tau_3\right).
\end{align*}
We then get the following theorem.
\begin{theorem}\label{unknownrefsignal}
Under the conditions of Theorem~\ref{FCLTNull} and if the reference signal fulfills (\ref{timeconstant}), we have under the null hypothesis
\[
\mathcal{F}^c_n(s,t,r_2,r_3)\Rightarrow\mathcal{F}^c(s,t,r_2,r_3),
\]
as $\min_{1\leq i\leq 3}n_i\to\infty$ for $0<s_0\leq s\leq 1$, $t\in[0,\overline{\tau}_1]$, $r_2\in[0,\overline{\tau}_2]$, and $r_3\in[0,\overline{\tau}_3]$.

The limit stochastic process $\mathcal{F}^c(s,t,r_2,r_3)$ is of the form
\begin{align*}
\mathcal{F}^c(s,t,r_2,r_3)\coloneqq\sqrt{\overline{\tau}_1\overline{\tau}_2\overline{\tau}_3}\sigma\int_0^s\int_0^1\int_0^1\!\varphi\left(t-\overline{\tau}_1z_1,r_2-\overline{\tau}_2z_2,r_3-\overline{\tau}_3z_3\right)\mathrm{d}B^c(z_1,z_2,z_3),
\end{align*}
where 
\begin{align*}
B^c(z_1,z_2,z_3)\coloneqq B(z_1,z_2,z_3)-\frac{z_1}{s_0}B(s_0,z_2,z_3)
\end{align*}
for $z_1\in[s_0,1]$, $(z_2,z_3)\in[0,1]^2$.
\end{theorem}

\section{Simulation of the detection process}\label{sim}
In this section we want to investigate the performance of the global maximum norm detector defined in (\ref{globmvb}). Before we do so, we have to explain how we calculate an appropriate control limit $c_M$ such that we can guarantee that the asymptotic false alarm probability is smaller than a given $\alpha$. We first note that for the global maximum norm detector $\mathcal M_n$ a type one error occurs if $\mathcal M_n/n_1<1$. Next, we adapt Theorem~2 of \cite{PawlakSteland} to our situation in order to obtain a more explicit formula for the type one error. For that we assume condition~(\ref{learningsample}), namely that
\begin{align*}
f(t,r_2,r_3)=f_0(t,r_2,r_3),\ 0\leq t\leq s_0\overline{\tau}_1, 0<s_0<1,
\end{align*}
i.e.\ we assume that there is no initial change in the signal.
\begin{theorem}
Let the condition in (\ref{learningsample}) hold. Under the assumptions of Theorem~\ref{FCLTNull} we have
\begin{align*}
\quad\lim_{n\to\infty}P_0\left(\frac{\mathcal M_n}{n_1}<1\right)=P\left(\sup_{s_0<s\leq 1}\sup_{0\leq t\leq s\overline{\tau}_1}\sup_{r_2\in[0,\overline{\tau}_2],\atop r_3\in[0,\overline{\tau}_3]}\left|\mathcal{F}(s,t,r_2,r_3)\right|>c_M\right),
\end{align*}
where $\mathcal{F}(s,t,r_2,r_3)$ is a zero mean Gaussian process with the covariance function defined in (\ref{covfct}).
\end{theorem}
Because of this theorem we can ensure that the asymptotic false alarm probability is not greater than $\alpha\in (0,1)$, if we choose the control limit $c_M=c_M(\alpha)$ as the smallest $c_M$ such that
\begin{align*}
P\left(\sup_{s_0\leq s\leq 1}\sup_{0\leq t\leq s\overline{\tau}_1}\sup_{r_2\in[0,\overline{\tau}_2],\atop r_3\in[0,\overline{\tau}_3]}\left|\mathcal{F}(s,t,r_2,r_3)\right|>c_M\right)\leq\alpha.
\end{align*}
As it is, however, not easy to obtain a concrete formula for the distribution of  
\[
X=\sup_{s_0\leq s\leq 1}\sup_{0\leq t\leq s\overline{\tau}_1}\sup_{r_2\in[0,\overline{\tau}_2],\atop r_3\in[0,\overline{\tau}_3]}\left|\mathcal{F}(s,t,r_2,r_3)\right| 
\]
we propose the following Monte Carlo algorithm to simulate $X$ and the control limit $c_M$. The algorithm is an adaption of the proposed algorithm in \cite{PawlakSteland}, p.~8, to the multivariable process $\mathcal{F}(s,t,r_2,r_3)$.

\textit{Step 1}: Generate trajectories of the Gaussian process $\mathcal{F}(s,t,r_2,r_3)$ on a grid $\{(s_i,t_j,(r_2)_k,(r_3)_l): i,j,k,l=1,\ldots,N\}$ where $0\leq s_1<\ldots<s_N\leq 1$, $0\leq t_1<\ldots<t_N\leq \overline{\tau}_1$, $0\leq (r_2)_1<\ldots<(r_2)_N\leq \overline{\tau}_2$ and $0\leq(r_3)_1<\ldots<(r_3)_N\leq \overline{\tau}_3$ for some $N\in\mathbb{N}$.

\textit{Step 2}: Return $X$ by calculating the maximum of the values $|\mathcal{F}(s_i,t_j,(r_2)_k,(r_3)_l)|$ for all $(i,j,k,l)$ such that the constraints $s_0\leq s_i\leq 1$ and $0\leq t_j\leq s_i\overline{\tau}_1$ are satisfied.

\textit{Step 3}: Using a large numer of repetitions of Step 1 and Step 2 produce realizations of $X$ to determine the empirical $(1-\alpha)-$quantile as an approximation for $c_M(\alpha)$.

We now begin our investigations with an illustrative example of the detection scheme. For that we assume that our reference signal is given by
\[
f_0(t,r_2,r_3)=\sin(6t)\sin(4r_2)\sin(4r_3)
\]
on $[0,2]^3$. Moreover, we assume that at the point in time $t=1$ a jump occurs over the whole image sequence which leads to an alternative signal of the form
\begin{align*}
f_1(t,r_2,r_3)=\begin{cases}
  f_0(t,r_2,r_3),  & t<1\\
  0.2+f_0(t,r_2,r_3), & t\geq 1
\end{cases}
\end{align*}
with $(t,r_2,r_3)\in[0,2]^3$. 
Thus, we obtain our noisy sample by the model
\[
y_{i_1,i_2,i_3}=f_j(i_1\tau_1,i_2\tau_2,i_3\tau_3)+\varepsilon_{i_1,i_2,i_3},
\]
for $j=0$ (null hypothesis) resp. $j=1$ (alternative),
where $\{\varepsilon_{i_1,i_2,i_3}\}$ is an i.i.d.\ $\mathcal N(0,1)$-distributed random field. In the following we take $\tau_1=0.04$ and $\tau_2=\tau_3=0.05$ corresponding to $n_1=\overline{\tau}_1/\tau_1=50$ observations in the time domain, and to $n_2=\overline{\tau}_2/\tau_2=40$ and $n_3=\overline{\tau}_3/\tau_3=40$ observations in the spatial domain. Moreover, we take the bandwidths as $\Omega_1=\Omega_2=\Omega_3=10$ and choose $s_0=0.05$ leading to $n_0=\left\lfloor s_0n_1\right\rfloor=2$.

We now consider the global maximum norm detector 
\[
\max_{0\leq t\leq \overline{\tau}_1 k/n_1} \max_{r_2\in[0,\overline{\tau}_2], r_3\in[0,\overline{\tau}_3]}\left|\mathcal{F}_n\left(\frac{k}{n_1},t,r_2,r_3\right)\right|
\]
for $k=n_0,\ldots,n_1$. If we take $\alpha=0.05$ and apply the Monte Carlo algorithm from above, we obtain as value for the control limit $c_M=23.3147$ which is the horizontal line in Figure~\ref{IllustrativeEx}. Furthermore, the solid line corresponds to the detection process under the alternative whereas the dashed line corresponds to the detection process under the null hypothesis. We can see that the partial sum process stays below the control limit for the whole observation period $[0,2]$ if there is no change in the signal. If we have, however, a change-point at $t=1$ the detection process directly reacts and crosses the control limit a short while later, namely for $k=30$ corresponding to the point in time $t_0=30\cdot 0.04=1.2$.
\begin{figure}
  \begin{center} 
  \includegraphics[width=120mm]{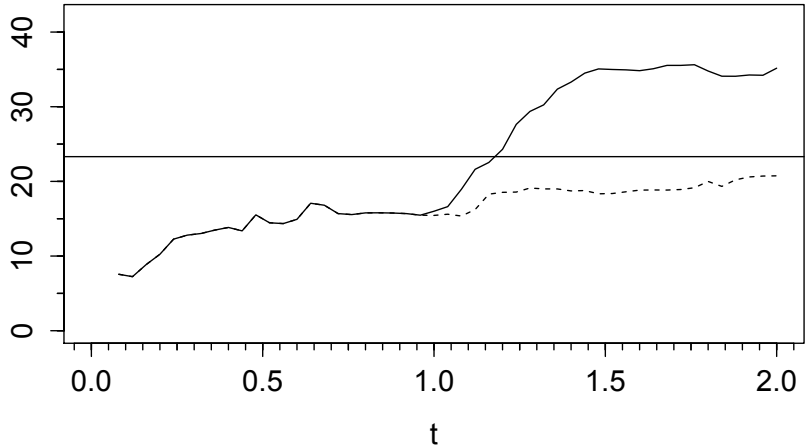}
  \caption{The global maximum norm detector $M_n$ applied to the signal $f_1(t,r_2,r_3)$ with a change-point at $t=1$. The change is detected at $t=1.2$.}
  \label{IllustrativeEx}
  \end{center}
\end{figure}

In the following simulation study we investigate the accuracy of the global maximum norm detector. Moreover, we evaluate the influence of different sampling periods and different correlation structures of the noise process. We also want to find out the influence of the asymptotic variance and its estimator developed in \cite{PrauseSteland2016} on the proper selection of the control limit $c_M$.

\subsection{Influence of the Sampling Periods}
We begin by analyzing the influence of different sampling periods in the spatial and time domain with respect to the rejection rates. For that, we calculate the corresponding control limit $c_M$ with the help of the Monte Carlo algorithm described above. Thus, we evaluate the process $\mathcal{F}(s,t,r_2,r_3)$ on the grid $\{(s_i,t_j,(r_2)_k,(r_3)_l): i,j,k,l=1,\ldots,N\}$ with $N=15$. After calculating the required maxima of $|\mathcal{F}(s_i,t_j,(r_2)_k,(r_3)_l)|$ we use the $95\%$-quantile of 10000 simulation replicates to estimate $c_M$.

In the following we adapt the setting of the illustrative example. As the noise process is modelled by an i.i.d.\ $\mathcal N(0,1)$-distributed random field, we obtain for the asymptotic variance $\sigma^2=1$.

Table~\ref{taberror1kind1000} shows the simulated type one errors for various sampling periods $\tau_1, \tau_2$, and $\tau_3$ for 1000 repetitions. We can see that the simulated rejection rates lie between 0.055 and 0.084 and thus that there is only a small influence of the sampling periods on the accuracy of the detector.
\begin{table}[h!]
\centering
\begin{tabular}{crrrrrrr}
\hline
$\tau_1$                       & $0.025$ & $0.0\overline{3}$ & $0.04$ & $0.05$ & $0.0\overline{6}$ & $0.08$ & $0.1$    \\
\hline
\multicolumn{1}{l|}{$\tau_2=\tau_3=0.0\overline{3}$}& 0.074 & 0.064 & 0.067 & 0.079 & 0.076 & 0.070 & 0.055 \\
\multicolumn{1}{l|}{$\tau_2=\tau_3=0.04$ }          & 0.074 & 0.066 & 0.067 & 0.061 & 0.064 & 0.076 & 0.062 \\
\multicolumn{1}{l|}{$\tau_2=\tau_3=0.05$}           & 0.074 & 0.069 & 0.084 & 0.070 & 0.063 & 0.066 & 0.058 \\
\hline
\end{tabular}
\caption{Simulated rejection rates for various sampling periods $\tau_1, \tau_2$ and $\tau_3$ for 1000 repetitions.}
\label{taberror1kind1000}
\end{table}

If we use 10000 instead of 1000 repetitions, we get even more accurate results. On account of very high computational costs, however, we only examined the rejection rates for 10000 repetitions in four cases.
\begin{table}[h!]
\centering
\begin{tabular}{crrrrrrr}
\hline
$\tau_1$                       & $0.08$ & $0.1$    \\
\hline
\multicolumn{1}{l|}{$\tau_2=\tau_3=0.04$}           & 0.0652 & 0.0598 \\
\multicolumn{1}{l|}{$\tau_2=\tau_3=0.05$}           & 0.0587 & 0.0567 \\
\hline
\end{tabular}
\caption{Simulated rejection rates for various sampling periods $\tau_1, \tau_2$, and $\tau_3$ for 10000 repetitions.}
\label{taberror1kind10000}
\end{table}

\subsection{Influence of Noise Correlations}
In this subsection we investigate how the rejection rates behave when using model (M4) as in \cite{PrauseSteland2016} for the noise process instead of taking i.i.d.\ errors, i.e.\ we put for $\rho\in(-1,1)$
\begin{align*}
(\textnormal{M4})\quad \varepsilon_{t,i,j}=X_{t,i,j}+v_{t,i,j},\qquad
X_{t,i,j}=\rho X_{t-1,i,j}+u_{t,i,j}
\end{align*}
where 
$u_{t,i,j}\stackrel{i.i.d.}\sim\mathcal N(0,1)$ for all $t,i,j\in\mathbb{Z}$ and the $v_{t,i,j}$ follow for each fixed $t$ their model (M1) which means that
\begin{align*} 
(\textnormal{M1})\,\,\, v_{i,j}=v_{t,i,j}=&\coloneqq a_1\eta_{i-1,j-1}+a_2\eta_{i-1,j}+a_3\eta_{i-1,j+1}+a_4\eta_{i,j-1}\\&+a_5\eta_{i,j}+a_6\eta_{i,j+1}+a_7\eta_{i+1,j-1}+a_8\eta_{i+1,j}+a_9\eta_{i+1,j+1}
\end{align*}
for i.i.d.\ innovations $\eta_{i,j}$ with $\eta_{i,j}\sim\mathcal{N}(0,1)$ for all $i$ and $j$ and real weights $a_k,\ k=1,\ldots,9$, where $a_5=1$ and $a_k=a$ for $k\neq 5$. Moreover, we suppose that the $v_{t,i,j}$
are uncorrelated for different values of $t$ and that $u_{t_1,i_1,j_1}$ and $v_{t_2,i_2,j_2}$ are uncorrelated for all $t_1,t_2,i_1,i_2,j_1,j_2\in\mathbb{Z}$.

We now fix $\tau_1=0.04, \tau_2=\tau_3=0.05$, and $u_0=0.25$ leading to $n_0=12$ as size for the learning sample in the time domain. The rest of the setting stays the same as in the illustrative example. We allow the autoregressive parameter $\rho$ to vary over the set $\{0.1,0.2,0.3,0.4,0.5\}$ while the moving average parameter $a$ lies in the set $\{\-0.1,0.01,0.1,0.3,0.5\}$. By Theorem~\ref{FCLTNull} we now obtain the proper control limit via the formula $c_{M,\sigma}=\sigma c_M$ where $\sigma$ is the asymptotic standard deviation of the noise process and $c_M$ the control limit for i.i.d.\ error terms. If the dependence structure of the noise process is unknown, one has to replace $\sigma$ by a proper estimator, see \cite{PrauseSteland2016}, leading to a control limit of the form $c_{M,\widehat{\sigma}}=\widehat{\sigma}c_M$.

Table~\ref{taberror1kindDepTrueSigma} shows the rejection rates for 1000 repetitions for control limits calculated with the true $\sigma$. We see that the rejection rates of the detector are quite accurate over the whole set of considered parameters. Smaller values of $\rho$ and $a$, reflecting a weak dependence structure of the noise process, lead to higher rejection rates, while larger values of these parameters, reflecting a strong dependence of the error terms, lead to lower rejection rates. Moreover, we can see that in most cases the rejection rates decrease for fixed $a$ and growing $\rho$ as well as for fixed $\rho$ and growing $a$, where this decrease is greater for smaller than for larger values of $a$ and $\rho$ respectively.
\begin{table}[h!]
\centering
\begin{tabular}{crrrrrrr}
\hline
$\rho$    & $0.1$ & $0.2$ & $0.3$ & $0.4$ & $0.5$   \\
\hline
\multicolumn{1}{l|}{$a=-0.1$}    & 0.099 & 0.086 & 0.065 & 0.053 & 0.036 \\
\multicolumn{1}{l|}{$a=0.01$}   & 0.063 & 0.063 & 0.059 & 0.056 & 0.039 \\
\multicolumn{1}{l|}{$a=0.1$}   & 0.058 & 0.053 & 0.047 & 0.041 & 0.035 \\
\multicolumn{1}{l|}{$a=0.3$}   & 0.033 & 0.033 & 0.036 & 0.038 & 0.033 \\
\multicolumn{1}{l|}{$a=0.5$}   & 0.031 & 0.028 & 0.027 & 0.029 & 0.028 \\
\hline
\end{tabular}
\caption{Simulated rejection rates for 1000 repetitions for control limits calculated with the true $\sigma$ ($n_1=50, n_0=12$).}
\label{taberror1kindDepTrueSigma}
\end{table}

When the asymptotic variance is unknown we can use the estimators proposed in \cite{PrauseSteland2016} in order to calculate the proper control limits. In this paper the accuracy of the estimators is shown in an extensive simulation study for various models of the underlying random field.

\subsection{Power Study}
We now want to analyse the power of our detection scheme when allowing for alternatives of the form (\ref{local3}). For that we take the local departure models as
\begin{align}\label{simualternative}
f_{n_1,n_2,n_3}(t,r_2,r_3)-f_0(t,r_2,r_3)=\frac{\delta_i(t,r_2,r_3)}{n_1^{\beta_1}n_2^{\beta_2}n_3^{\beta_3}}
\end{align}
with 
\begin{align}
\delta_i(t,r_2,r_3)=g_i\sin\left(15(t-1)+\frac{\pi}{2}\right)\sin\left(4r_2\right)\sin\left(4r_3\right)
\end{align}
for $(t,r_2,r_3)\in[0,2]^3$ and $(g_1,g_2,g_3)=(68,125,559)$. Moreover, we put $\beta_2=\beta_3=0.5$ and choose $\beta_1\in\{0.3,0.5,1\}$. 



Note that the alternative signal $f_{n_1,n_2,n_3}(t,r_2,r_3)$ displays an amplitude, a frequency as well as a phase distortion in the time domain compared to the reference signal $f_0(t,r_2,r_3)=\sin(6t)\sin(4r_2)\sin(4r_3)$. 

By the asymptotic theory of Theorem~\ref{detalt} we expect that changes with $\beta_1>1/2$ cannot be detected for large sample sizes in the time domain, whereas changes with $\beta_1\leq1/2$ can easily be detected. In particular, the case $\beta_1=1/2$ corresponds to the case where the process $\mathcal{F}_n(s,t,r_2,r_3)$ converges to the process $\mathcal{F}^{\delta}(s,t,r_2,r_3)$ defined in Theorem~\ref{detalt}. 

The values of $g_i, i=1,2,3,$ were chosen in such a way that the initial power for $n_1=20$ observations in the time domain is reasonably high and the same for all three different values of $\beta_1$. Here, $g_1=68$ corresponds to $\beta_1=0.3$, $g_2=125$ to $\beta_1=0.5$, and $g_3=559$ to $\beta_1=1$.

The rest of the simulation setting stays the same as in the illustrative example; in particular, we suppose that the error terms are i.i.d.\ and that this is known. The resulting power curves are shown in Figure~\ref{PowerIID}.

\begin{figure}[h!]
\begin{center}
  \includegraphics[width=120mm]{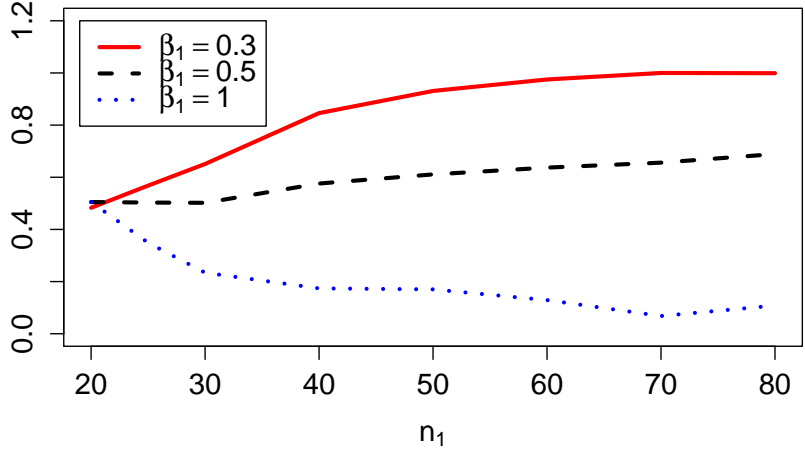}
  \caption{Simulated power for the models in (\ref{simualternative}) for different values of $\beta_1$ as a function of the sample size $n_1$ in the time domain.}
  \label{PowerIID}
\end{center}
\end{figure}

Due to very high computational costs the sample size $n_1$ only varies between $20$ and $80$.
We see that the curves indeed confirm what is predicted by the theory:\ For $\beta_1<1/2$ the power increases towards one and for $\beta_1>1/2$ it decreases towards the type one error rate of $\alpha=5\%$. For $\beta_1=1/2$ one could assume that the power increases as well when looking at Figure~\ref{PowerIID}; Table~\ref{beta105power} shows, however, that for larger sample sizes in the time domain it levels off between 0.7 and 0.8 which thus confirms the theory as well.

\begin{table}[h!]
\centering
\begin{tabular}{crrrrr}
\hline
$n_1$         & $80$ & $160$ & $240$ & $320$ \\
\hline
\multicolumn{1}{l|}{$\beta_1=0.5$} & 0.689 & 0.760 & 0.778 & 0.786 \\
\hline
\end{tabular}
\caption{Simulated power for the models in (\ref{simualternative}) for $\beta_1=0.5$ and different sample sizes $n_1$.}
\label{beta105power}
\end{table}

\appendix

\section{Proofs}

\subsection{Preliminaries}\label{details}
Before proving the main results of this paper we briefly want to review the Skorohod space for functions defined on $[0,1]^q$ as well as the multivariate Riemann-Stieltjes integral.

The functional $\mathcal F_n(s,t,r_2,r_3)$ and its variations that were introduced in the previous sections can be viewed as an element of the Skorohod space $D_q=D[0,1]^q$ for functions defined on $[0,1]^q$ with $q\in\mathbb{N}$. This function space is the generalisation of the well-known Skorohod space $D=D[0,1]$ for functions with a single time parameter to functions with several time parameters and thus allows for certain discontinuities.

Informally speaking, the space $D[0,1]^q$ contains all real-valued functions that are `continuous from above, with limits from below'. To be more precise, let $\bm t\in[0,1]^q$ and define $R_p$, for all $1\leq p\leq q$, as one of the relations `$<$' and `$\geq$'. Denote by $Q_{R_1,\ldots,R_q}(\bm t)$ the quadrant
\[
\{(s_1,\ldots,s_q)\in[0,1]^q: s_pR_pt_p, 1\leq p\leq q\}.
\]
We say that $f$ is an element of $D_q$ if the following two conditions hold:
\begin{itemize}
	\item[(a)] The limit \[f_Q\coloneqq\lim_{\bm s\to \bm t,\atop \bm s\in Q}f(\bm s)\] exists for each of the $2^q$ quadrants $Q=Q_{R_1,\ldots,R_q}(\bm t)$. 
	\item[(b)] We have \[f(\bm t)=f_{Q_{\geq,\ldots,\geq}}=\lim_{\bm s\to \bm t\atop \bm s\in Q_{\geq,\ldots,\geq}} f(\bm s).\]
\end{itemize}
These conditions are analogously to the ones for $D[0,1]$. 
For more details, including the notion of weak convergence in $D[0,1]^q$, we refer to \cite{Straf} and \cite{BickelWichura} respectively.

The generalization of the univariate Riemann-Stieltjes integral to the multivariate Riemann-Stieltjes integral in $q$ dimensions is now done by means of the $q$-fold alternating sum of a function in (\ref{alternierendeSumme}), cf. also \cite{Sard}.

Let $\mathcal Y=\prod_{j=1}^q\mathcal Y_j$ be a ladder on the hyperrectangle $[\bm{a},\bm{b}]\subset\mathbb R^q$, where $\mathcal Y_j$ is a ladder on $\left[a_j,b_j\right]$ for all $j=1,\ldots,q$.  Let $\widetilde{\bm y}\in\left[\bm{y},\bm{y}^+\right]$ for each $\bm y\in\mathcal Y$, where $\bm y^+$ denotes the component-by-component successor of $\bm y$, see above. Define a norm $\|\mathcal Y\|$ on $[\bm{a},\bm{b}]$ as
\[
\left\|\mathcal Y\right\|\coloneqq\max\left\{\max_{y_1\in\mathcal Y_1}\left(y_1^+-y_1\right),\max_{y_2\in\mathcal Y_2}\left(y_2^+-y_2\right),\ldots,\max_{y_q\in\mathcal Y_q}\left(y_q^+-y_q\right)\right\}.
\]
Suppose that $f$ and $h$ are real-valued functions defined on the hyperrectangle $[\bm{a},\bm{b}]$.
Now, for each ladder $\mathcal Y\in\mathbb Y=\prod_{j=1}^q\mathbb Y_j$ consider the so-called Riemann-Stieltjes sum\index{Riemann-Stieltjes sum}
\begin{align}\label{RSSumme}
\Sigma
\coloneqq\sum_{\bm y\in\mathcal Y}h\left(\widetilde{\bm y}\right)\Delta\left(f;\bm y,\bm y_+\right)
=\sum_{\bm y\in\mathcal Y}h\left(\widetilde{\bm y}\right)\left[\sum_{v\subseteq\{1,\ldots,q\}}(-1)^{|v|}f\left(\bm y_v:\bm y^+_{-v}\right)\right].
\end{align}
Analogously to the one-dimensional case we now define the Riemann-Stieltjes integral of $h$ with respect to $f$ as
\begin{align}\label{DefMRSI}
\int_{\bm a}^{\bm b}\!h(\bm{x})\,\mathrm{d}f(\bm{x})\coloneqq\lim_{\left\|\mathcal Y\right\|\to 0}\Sigma,
\end{align}
if the latter exists. The integral on the left is understood as a multivariate integral, namely as
\begin{align*}
\int_{\bm a}^{\bm b}\!h(\bm{x})\,\mathrm{d}f(\bm{x})
&=\int_{\bm{a}_{1:q}}^{\bm{b}_{1:q}}\!h(\bm{x})\,\mathrm{d}f\left(\bm{x}\right)\\&
=\int_{a_1}^{b_1}\int_{a_2}^{b_2}\ldots\int_{a_q}^{b_q}\!h\left(x_1,x_2,\ldots,x_q\right)\,\mathrm{d}f\left(x_1,x_2,\ldots,x_q\right).
\end{align*}
Similar to the one-dimensional case this integral and all `lower dimensional' integrals exist if $h$ is a continuous function on the hyperrectangle $[\bm a,\bm b]$, $\bm a,\bm b\in\mathbb{R}^q$ and if $f$ is a function of bounded variation in the sense of Hardy and Krause on $[\bm a,\bm b]$.

Moreover, there exists a generalization of the integration by parts formula for multivariate Riemann-Stieltjes integrals, see \cite{Young}, p.~287.
This allows us to define the multivariate Riemann-Stieltjes integral even with respect to functions $h$ that are not of bounded variation in the sense of Vitali or in the sense of Hardy and Krause. In this case we take the integration by parts formula as a definition for the integral, i.e.\ we put
\begin{align}\label{DefunbeschrVar}
\int_{\bm{a}_{1:q}}^{\bm{b}_{1:q}}\!f(\bm{x})\,\mathrm{d}h(\bm{x})&
\coloneqq\left[h(\bm{x})f(\bm{x})\right]_{\bm{a}_{1:q}}^{\bm{b}_{1:q}}+(-1)^q\int_{\bm{a}_{1:q}}^{\bm{b}_{1:q}}\!h(\bm{x})\,\mathrm{d}f(\bm{x})\notag\\&
+\sum_{v\subseteq\{1,\ldots,q\},\atop 1\leq |v|\leq q-1}(-1)^{|v|}\int_{\bm{a}_v}^{\bm{b}_v}\!\left[h(\bm{x})\,\mathrm{d}f(\bm{x})\right]_{\bm{a}_{-v}}^{\bm{b}_{-v}}
\end{align}
whenever the right-hand side exists.
The square bracket notation is used for the evaluation, respectively, the increment of a multivariate antiderivative $f$ over a hyperrectangle $[\bm{a},\bm{b}]\subset\mathbb R^q$. Thus, we have
\[
\left[f(\bm{x})\right]_{\bm a}^{\bm b}
=\left[f(x_1,\ldots,x_q)\right]_{\bm{a}_{1:q}}^{\bm{b}_{1:q}}
\coloneqq\Delta(f;\bm a,\bm b)
=\sum_{v\subseteq\{1,\ldots,q\}}(-1)^{|v|}f\left(\bm{a}_v:\bm{b}_{-v}\right).
\]
If the evaluation of $f$ only takes place over a hyperrectangle $\left[\bm{a}_{-v},\bm{b}_{-v}\right],\ v\subset 1:q$, we write $\left[f(\bm{x})\right]_{\bm{a}_{-v}}^{\bm{b}_{-v}}$, and this is defined as
\[
\left[f(\bm{x})\right]_{\bm{a}_{-v}}^{\bm{b}_{-v}}
=\left[f\left(\bm{x}_v:\bm{x}_{-v}\right)\right]_{\bm{a}_{-v}}^{\bm{b}_{-v}}
\coloneqq\sum_{w\subseteq -v}(-1)^{|w|}f\left(\bm{x}_v:\bm a_w:\bm b_{-w}\right).
\]
Thus, we can define the Riemann-Stieltjes integral with respect to a (multivariable) Brownian motion, where the sample paths are of unbounded variation almost surely.

\subsection{Proofs}
This section is devoted to rigorous proofs of the presented theorems. Note that parts of the material are taken from \cite{Prause}.
We start with proving Theorem~\ref{FCLTNull} and introduce the following abbreviating notation.
Let
\begin{align}\label{defphi_w}
\varphi_w(v_1,v_2,v_3,t,r_2,r_3)\coloneqq
\varphi\left(t-\overline{\tau}_1v_1,r_2-\overline{\tau}_2v_2,r_3-\overline{\tau}_3v_3\right)
\end{align}
for $v_1,v_2,v_3\in [0,1]$, $t\in[0,\overline{\tau}_1]$, $r_2\in[0,\overline{\tau}_2]$ and $r_3\in[0,\overline{\tau}_3]$. Here, 
\begin{align*}
&\quad\ \ \varphi\left(t-\overline{\tau}_1v_1,r_2-\overline{\tau}_2v_2,r_3-\overline{\tau}_3v_3\right)\\&
\hspace{2cm}=\frac{\sin(\Omega_1 (t-\overline{\tau}_1v_1))}{\pi (t-\overline{\tau}_1v_1)}\frac{\sin(\Omega_2 (r_2-\overline{\tau}_2v_2))}{\pi (r_2-\overline{\tau}_2v_2)}\frac{\sin(\Omega_3 (r_3-\overline{\tau}_3v_3))}{\pi (r_3-\overline{\tau}_3v_3)}\\&
\hspace{2cm}\eqqcolon\varphi_1(v_1,t)\varphi_2(v_2,r_2)\varphi_3(v_3,r_3)
\end{align*}
with $\varphi_j(v_j,\overline{\tau}_jv_j)=\Omega_j/\pi$ for $j=1,2,3$.

For the proof of Theorem~\ref{FCLTNull} we need the following lemmas.
\begin{lemma}\label{phi_wBVHK}
Let $\varphi_w$ be defined as in (\ref{defphi_w}). Then
\[
\sup_{t\in[0,\overline{\tau}_1],\atop{r_2\in[0,\overline{\tau}_2],r_3\in[0,\overline{\tau}_3]}}
V_{HK}(\varphi_w(\cdot,\cdot,\cdot,t,r_2,r_3))<\infty.
\]
\end{lemma}
\begin{proof}
By Proposition~11 in \cite{Owen} we know that for $f,g\in\ $BVHK we also have $fg\in\ $BVHK; thus, we can consider each factor of $\varphi_w$ separately and the assertion follows by the one-dimensional theory for functions of bounded variation.
\end{proof}
\begin{lemma}\label{PIMEstimateVHK}
Let $h$ be a real-valued function on $[\bm{a},\bm{b}]$ and let $f$ be a function of bounded variation in the sense of Hardy and Krause on $[\bm{a},\bm{b}]$. Suppose that
\[
\int_{\bm{a}_v}^{\bm{b}_v}\!\left[f(\bm{x})\,\mathrm{d}h(\bm{x})\right]_{\bm{a}_{-v}}^{\bm{b}_{-v}}
\]
exists for all $\emptyset\neq v\subseteq 1:q$.
Then there exists a constant $c\in\mathbb R$ such that
\begin{align}\label{ComibPIMVar}
\left|\int_{\bm a}^{\bm b}\!f(\bm{x})\,\mathrm{d}h(\bm{x})\right|
\leq 2^q\sup_{\bm{x}\in[\bm{a},\bm{b}]}|h(\bm{x})|\left[\sup_{\bm{x}\in[\bm{a},\bm{b}]}|f(\bm{x})|+cV_{HK}(f)\right].
\end{align}
\end{lemma}
\begin{proof}
The assertion follows by an application of the integration by parts formula combined with the definition of the multivariate Riemann-Stieltjes integral and the variation in the sense of Hardy and Krause.
\end{proof}
\begin{lemma}\label{lemmaSI}
Let $\bm a,\bm b\in\mathbb R^q$, $\widetilde{\bm a},\widetilde{\bm b}\in\mathbb R^p$ with $\bm a\leq \bm b,\widetilde{\bm a}\leq\widetilde{\bm b}$, $p,q\in\mathbb N$.
Let $\psi$ be a bounded function on $[\bm a,\bm b]\times[\widetilde{\bm a},\widetilde{\bm b}]$ with
\[
\sup_{\bm y\in[\widetilde{\bm a},\widetilde{\bm b}]}V_{HK}\left(\psi(\bm{\cdot},\bm y)\right)<\infty.
\]
Let $\{f_n\}$ be a sequence of functions such that $f_n\in D[\bm a,\bm b]$ for all $n\in\mathbb N$. Suppose that $\lim_{n\to\infty}f_n=f$ in the Skorohod topology for a function $f\in C[\bm a,\bm b]$. Moreover, suppose that
\[
\int_{\bm a_v}^{\bm s_v}\!\left[\psi(\bm x,\bm y)\,\mathrm{d}f_n(\bm x)\right]_{\bm a_{-v}}^{\bm s_{-v}}
\]
exists for $\bm s_v\in[\bm a_v,\bm b_v]$, $\bm y\in[\widetilde{\bm a},\widetilde{\bm b}]$ and for all $n\in\mathbb N$ and $\emptyset\neq v\subseteq 1:q$. Then,
\[
\lim_{n\to\infty}\sup_{\bm s\in[\bm a,\bm b]}\sup_{\bm y\in[\widetilde{\bm a},\widetilde{\bm b}]}\left|\int_{\bm a}^{\bm s}\!\psi(\bm x,\bm y)\,\mathrm{d}f_n(\bm x)-\int_{\bm a}^{\bm s}\!\psi(\bm x,\bm y)\,\mathrm{d}f(\bm x)\right|=0.
\]
\end{lemma}
\begin{proof}
An application of Lemma \ref{PIMEstimateVHK} and the fact that Skorohod convergence towards a continuous limit implies uniform convergence show the assertion.
\end{proof}
We now prove Theorem \ref{FCLTNull}. To simplify the notation we put $n=n_1=n_2=n_3$, but the proof can be completed along the same lines if the $n_i, i=1,2,3,$ differ.
\begin{proof}[Proof of Theorem \ref{FCLTNull}]
We consider the function space $D[0,1]^3$ and define the functional
\[
\Lambda(f)=\Lambda(f)(s,t,r_2,r_3)\coloneqq\sqrt{\overline{\tau}_1\overline{\tau}_2\overline{\tau}_3}
\int_0^s\int_0^1\int_0^1\!\varphi_w(v_1,v_2,v_3,t,r_2,r_3)\,\mathrm{d}f(v_1,v_2,v_3)
\]
for $0<s_0\leq s\leq 1$, $t\in[0,\overline{\tau}_1]$, $r_2\in[0,\overline{\tau}_2]$, and $r_3\in[0,\overline{\tau}_3]$
on it, whenever the integral exists. Here $\varphi_w$ is defined as in (\ref{defphi_w}).
Let $\{f_n\}$ be a sequence of functions such that $f_n\in D[0,1]^3$ for all $n\in\mathbb N$.
We know by Lemma~\ref{lemmaSI}, that $\lim_{n\to\infty}f_n=f$ in the Skorohod topology with $f\in C[0,1]^3$ implies
\begin{align}\label{convfctl}
&\sup_{s\in[s_0,1]}\sup_{t\in[0,\overline{\tau}_1],\atop r_2\in[0,\overline{\tau}_2],r_3\in[0,\overline{\tau}_3]}\left|\int_0^s\int_0^1\int_0^1 \!\varphi_w(v_1,v_2,v_3,t,r_2,r_3)\,\mathrm{d}f_n(v_1,v_2,v_3)\right.\notag\\&\left.\hspace{3.5cm}
-\int_0^s\int_0^1\int_0^1\!\varphi_w(v_1,v_2,v_3,t,r_2,r_3)\,\mathrm{d}f(v_1,v_2,v_3)\right|\to 0
\end{align}
in $D([s_0,1]\times[0,\overline{\tau}_1]\times[0,\overline{\tau}_2]\times[0,\overline{\tau}_3])$ as $n\to\infty$, since 
\[
\sup_{t\in[0,\overline{\tau}_1],\atop{r_2\in[0,\overline{\tau}_2],r_3\in[0,\overline{\tau}_3]}}
V_{HK}(\varphi_w(\cdot,\cdot,\cdot,t,r_2,r_3))<\infty
\]
by Lemma~\ref{phi_wBVHK}.
Assumption~1 and the continuous mapping theorem now lead to
\begin{align*}
\Lambda(Z_n(\cdot))&=\sqrt{\overline{\tau}_1\overline{\tau}_2\overline{\tau}_3}
\int_0^s\int_0^1\int_0^1\!\varphi_w(v_1,v_2,v_3,t,r_2,r_3)\,\mathrm{d}Z_n(v_1,v_2,v_3)\\&
\Rightarrow\sqrt{\overline{\tau}_1\overline{\tau}_2\overline{\tau}_3}\,\sigma
\int_0^s\int_0^1\int_0^1\!\varphi_w(v_1,v_2,v_3,t,r_2,r_3)\,\mathrm{d}B(v_1,v_2,v_3),
\end{align*}
as $n\to\infty$, since $P(B(v_1,v_2,v_3)\in C\left([0,1]^3\right))=1$.

On the other hand we have
\begin{align}\label{form1}
\Lambda(Z_n(\cdot))=\sqrt{\frac{\overline{\tau}_1\overline{\tau}_2\overline{\tau}_3}{n^3}}
\int_0^s\int_0^1\int_0^1\!\varphi_w(v_1,v_2,v_3,t,r_2,r_3)\,\mathrm{d}\left(\sum_{i_1=1}^{\lfloor nv_1\rfloor}\sum_{i_2=1}^{\lfloor nv_2\rfloor}
\sum_{i_3=1}^{\lfloor nv_3\rfloor}\varepsilon_{i_1,i_2,i_3}\right).
\end{align}
We now interpret this last integral as a multivariate Riemann-Stieltjes integral. Let $\mathcal{Y}_1$ be a ladder on $[0,s]$ for $s\in[s_0,1]$ and let $\mathcal{Y}_2$ and $\mathcal{Y}_3$ be ladders on $[0,1]$. Then we can write the triple integral as
\begin{align*}
&\quad\int_0^s\int_0^1\int_0^1\!\varphi_w(v_1,v_2,v_3,t,r_2,r_3)\,\mathrm{d}\left(\sum_{i_1=1}^{\lfloor nv_1\rfloor}\sum_{i_2=1}^{\lfloor nv_2\rfloor}
\sum_{i_3=1}^{\lfloor nv_3\rfloor}\varepsilon_{i_1,i_2,i_3}\right)\\&
=
\lim_{\left\|\mathcal Y\right\|\to 0}\sum_{\bm y\in\mathcal Y}
\varphi_w(\widetilde{y}_1,\widetilde{y}_2,\widetilde{y}_3,t,r_2,r_3)\\&\hspace{2.2cm}
\left[
\sum_{i_1=1}^{\lfloor ny^+_1\rfloor}\sum_{i_2=1}^{\lfloor ny^+_2\rfloor}\sum_{i_3=1}^{\lfloor ny^+_3\rfloor}\varepsilon_{i_1,i_2,i_3}
-\sum_{i_1=1}^{\lfloor ny_1\rfloor}\sum_{i_2=1}^{\lfloor ny^+_2\rfloor}\sum_{i_3=1}^{\lfloor ny^+_3\rfloor}\varepsilon_{i_1,i_2,i_3}\right.\\&\left.\hspace{2.2cm}
-\sum_{i_1=1}^{\lfloor ny^+_1\rfloor}\sum_{i_2=1}^{\lfloor ny_2\rfloor}\sum_{i_3=1}^{\lfloor ny^+_3\rfloor}\varepsilon_{i_1,i_2,i_3}
-\sum_{i_1=1}^{\lfloor ny^+_1\rfloor}\sum_{i_2=1}^{\lfloor ny^+_2\rfloor}\sum_{i_3=1}^{\lfloor ny_3\rfloor}\varepsilon_{i_1,i_2,i_3}\right.\\&\left.\hspace{2.2cm}
+\sum_{i_1=1}^{\lfloor ny_1\rfloor}\sum_{i_2=1}^{\lfloor ny_2\rfloor}\sum_{i_3=1}^{\lfloor ny^+_3\rfloor}\varepsilon_{i_1,i_2,i_3}
+\sum_{i_1=1}^{\lfloor ny_1\rfloor}\sum_{i_2=1}^{\lfloor ny^+_2\rfloor}\sum_{i_3=1}^{\lfloor ny_3\rfloor}\varepsilon_{i_1,i_2,i_3}\right.\\&\left.\hspace{2.2cm}
+\sum_{i_1=1}^{\lfloor ny^+_1\rfloor}\sum_{i_2=1}^{\lfloor ny_2\rfloor}\sum_{i_3=1}^{\lfloor ny_3\rfloor}\varepsilon_{i_1,i_2,i_3}
-\sum_{i_1=1}^{\lfloor ny_1\rfloor}\sum_{i_2=1}^{\lfloor ny_2\rfloor}\sum_{i_3=1}^{\lfloor ny_3\rfloor}\varepsilon_{i_1,i_2,i_3}
\right],
\end{align*}
where $\bm y_+=\left(y^+_1,y^+_2,y^+_3\right)$ is the component-by-component successor of the point $\bm y=\left(y_1,y_2,y_3\right)$, $\widetilde{\bm y}=\left(\widetilde{y}^1,\widetilde{y}^2,\widetilde{y}^3\right)$ is an arbitrary point in the cube $[\bm y,\bm y_+]$ and
\[
\left\|\mathcal Y\right\|=\max\left\{\max_{y_1\in\mathcal{Y}_1}(y^+_1-y_1), \max_{y_2\in\mathcal{Y}_2}(y^+_2-y_2), \max_{y_3\in\mathcal{Y}_3}(y^+_3-y_3)\right\}.
\]
Consider now, without loss of generality, a ladder $\mathcal Y=\prod_{i=1}^3\mathcal Y_i$ with $y_i<k_i/n\leq y^+_i$ for all $k_1\in\{1,\ldots,\left\lfloor ns\right\rfloor\}$ and all $k_2,k_3\in\{1,\ldots,n\}$, respectively, and write $\widetilde{y}_{k_i}$ for points $\widetilde{y}_{k_i}\in(y_i,y_i^+], i=1,2,3$.
As the floor function $\lfloor ny\rfloor$ is constant on intervals of the form $\left[(k-1)/n,k/n\right)$, $k\in\{1,\ldots,n\}$, the triple sums in the last expression can be combined and we obtain
\begin{align*}
&\quad\sum_{i_1=1}^{\lfloor ny^+_1\rfloor}\sum_{i_2=1}^{\lfloor ny^+_2\rfloor}\varepsilon_{i_1,i_2,k_3}
-\sum_{i_1=1}^{\lfloor ny^+_1\rfloor}\sum_{i_2=1}^{\lfloor ny_2\rfloor}\varepsilon_{i_1,i_2,k_3}
-\sum_{i_1=1}^{\lfloor ny_1\rfloor}\sum_{i_2=1}^{\lfloor ny^+_2\rfloor}\varepsilon_{i_1,i_2,k_3}
+\sum_{i_1=1}^{\lfloor ny_1\rfloor}\sum_{i_2=1}^{\lfloor ny_2\rfloor}\varepsilon_{i_1,i_2,k_3}\\&
=
\sum_{i_1=1}^{\lfloor ny^+_1\rfloor}\varepsilon_{i_1,k_2,k_3}
-\sum_{i_1=1}^{\lfloor ny_1\rfloor}\varepsilon_{i_1,k_2,k_3}
=
\varepsilon_{k_1,k_2,k_3}.
\end{align*}
This leads to
\begin{align*}
&\quad
\int_0^s\int_0^1\int_0^1\!\varphi_w(v_1,v_2,v_3,t,r_2,r_3)\,\mathrm{d}\left(\sum_{i_1=1}^{\lfloor nv_1\rfloor}\sum_{i_2=1}^{\lfloor nv_2\rfloor}
\sum_{i_3=1}^{\lfloor nv_3\rfloor}\varepsilon_{i_1,i_2,i_3}\right)\\&
=
\lim_{\left\|P\right\|\to 0}\sum_{k_1=1}^{\lfloor ns\rfloor}\sum_{k_2=1}^{n}\sum_{k_3=1}^{n}
\varphi_w(\widetilde{y}_{k_1},\widetilde{y}_{k_2},\widetilde{y}_{k_3},t,r_2,r_3)\varepsilon_{k_1,k_2,k_3}\\&
=
\sum_{k_1=1}^{\lfloor ns\rfloor}\sum_{k_2=1}^{n}\sum_{k_3=1}^{n}
\varphi_w\left(\frac{k_1}{n},\frac{k_2}{n},\frac{k_3}{n},t,r_2,r_3\right)\varepsilon_{k_1,k_2,k_3},
\end{align*}
where the last equality holds
as
\begin{align*}
\left\|\left(\widetilde{y}_{k_1},\widetilde{y}_{k_2},\widetilde{y}_{k_3}\right)^{\prime}
-\left(\frac{k_1}{n},\frac{k_2}{n},\frac{k_3}{n}\right)^{\prime}\right\|
&\leq
\left\|\left(y^+_1-y_1,y^+_2-y_2,y^+_3-y_3\right)^{\prime}\right\|\\&
\leq \left\|\mathcal Y\right\|\to 0.
\end{align*}
Now we can write (\ref{form1}) as
\begin{align*}
\Lambda(Z_n(\cdot))&=\sqrt{\frac{\overline{\tau}_1\overline{\tau}_2\overline{\tau}_3}{n^3}}
\sum_{k_1=1}^{\lfloor ns\rfloor}\sum_{k_2=1}^{n}\sum_{k_3=1}^{n}
\varphi_w\left(\frac{k_1}{n},\frac{k_2}{n},\frac{k_3}{n},t,r_2,r_3\right)\varepsilon_{k_1,k_2,k_3}.
\end{align*}
If we recall the definition of $\varphi_w$ from (\ref{defphi_w}) and the fact that $\tau_j\approx\overline{\tau}_j/n$, $j=1,2,3$, 
we obtain that
\[
\mathcal{F}_n(s,t,r_2,r_3)=\Lambda(Z_n(\cdot))+o_P(1).
\]
This completes the proof.
\end{proof}
\begin{proof}[Proof of Lemma \ref{covstructure}]
The proof of this lemma is omitted here as it is analogously to the proof of Corollary~1 in \cite{PawlakSteland}, since the properties in Theorem~5.1.4 in \cite{AshGardner} also hold true for multivariate Riemann-Stieltjes integrals.
\end{proof}
\begin{proof}[Proof of Lemma \ref{cltdet}]
The assertion directly follows with Theorem \ref{FCLTNull} and the continuous mapping theorem.
\end{proof}
Before we prove Theorem \ref{detalt} we review the Hwlaka-Koksma inequality which gives an error bound for the discrete approximation of a Riemann integral for functions of bounded variation in the sense of Hardy and Krause, see for example \cite{Niederreiter}, Theorem~2.11. For this we need the notion of discrepancy which measures the deviation of an arbitrary point set in the unit cube from a uniformly distributed point set. Hence, let $P$ be a point set consisting of $\bm x^{(1)},\ldots,\bm x^{(N)}\in[0,1]^q$. Now, for an arbitrary subset $B$ of $[0,1]^q$, we define
\[
A(B;P)\coloneqq\sum_{i=1}^N\bm{1}_B\left(\bm x^{(i)}\right),
\]
i.e.\ $A(B;P)$ counts the number of points of $\bm x^{(1)},\ldots,\bm x^{(N)}$ lying in $B$. We then define the general discrepancy of the point set $P$ for a nonempty family $\mathcal B$ of Lebesgue-measurable subsets of $[0,1]^q$ as
\[
D_N(B;P)\coloneqq\sup_{B\in\mathcal B}\left|\frac{A(B;P)}{N}-\lambda_q(B)\right|.
\]
This leads to the following notion of discrepancy, cf.\ Definition~2.1. in \cite{Niederreiter}.
  \begin{definition}
The \textit{star discrepancy} $D_N^{\star}(P)=D_N^{\star}\left(\bm x^{(1)},\ldots,\bm x^{(N)}\right)$ of the point set $P$ is defined by $D_N^{\star}(P)\coloneqq D_N(\mathcal J^{\star}; P)$, where $\mathcal J^{\star}$ is the family of all subintervals of $[0,1]^q$ of the form $\prod_{i=1}^q[0,u_i)$.
  \end{definition}
We can now state the well-known Hwlaka-Koksma inequality for multivariable functions.
  \begin{lemma}\label{HKinequality}
If $f$ has bounded variation $V_{HK}(f)$ on $[0,1]^q$ in the sense of Hardy and Krause, then, for any $\bm x^{(1)},\ldots,\bm x^{(N)}\in[0,1]^q$, we have
\[
\left|\frac{1}{N}\sum_{i=1}^Nf\left(\bm x^{(i)}\right)-\int_{[0,1]^q}\!f(\bm u)\,\mathrm{d}\bm u\right|\leq V_{HK}(f)D_N^{\star}\left(\bm x^{(1)},\ldots,\bm x^{(N)}\right).
\]
  \end{lemma}
We now prove Theorem \ref{detalt}. To simplify the notation we put again $n=n_1=n_2=n_3$, but the proof can be completed along the same lines if the $n_i, i=1,2,3,$ differ.
\begin{proof}[Proof of Theorem \ref{detalt}]
The local alternative is given by
\[
y_{k_1,k_2,k_3}=f_n\left(k_1\tau_1,k_2\tau_2,k_3\tau_3\right)+\varepsilon_{k_1,k_2,k_3}
\]
with 
\[
f_n(t,r_2,r_3)=f_0(t,r_2,r_3)+\frac{\delta(t,r_2,r_3)}{n^{\beta}}
\]
for $\beta>0$.
Our test statistic $\mathcal{F}_n(s,t,r_2,r_3)$ is therefore defined as
\begin{align*}
\mathcal{F}_n(s,t,r_2,r_3)
&=\sqrt{\tau_1\tau_2\tau_3}
\sum_{l_1=1}^{\lfloor ns\rfloor}\sum_{l_2=1}^{n}\sum_{l_3=1}^{n}\varepsilon_{k_1,k_2,k_3}
\varphi\left(t-k_1\tau_1,r_2-k_2\tau_2,r_3-k_3\tau_3\right)\\&
+\frac{\sqrt{\tau_1\tau_2\tau_3}}{n^{\beta}}\sum_{l_1=1}^{\lfloor ns\rfloor}\sum_{l_2=1}^{n}\sum_{l_3=1}^{n}\delta(k_1\tau_1,k_2\tau_2,k_3\tau_3)\\&\hspace{4.1cm}
\varphi\left(t-k_1\tau_1,r_2-k_2\tau_2,r_3-k_3\tau_3\right)\\&
\eqqcolon T_n^{(1)}(s,t,r_2,r_3)+T_n^{(2)}(s,t,r_2,r_3).
\end{align*}
$T_n^{(1)}(s,t,r_2,r_3)$ equals $\mathcal{F}_n(s,t,r_2,r_3)$ under the null hypothesis which converges to the process $\mathcal{F}(s,t,r_2,r_3)$ of Theorem~\ref{FCLTNull} for $n\to\infty$.

By assumption on the sampling periods we obtain for the second process
\begin{align*}
&T_n^{(2)}(s,t,r_2,r_3)
=\frac{1}{n^{\beta-3/2}}\frac{1}{\sqrt{\overline{\tau}_1\overline{\tau}_2\overline{\tau}_3}}\frac{\overline{\tau}_1\overline{\tau}_2\overline{\tau}_3}{n^3}\\&\qquad
\sum_{k_1=1}^{\lfloor ns\rfloor}\sum_{k_2=1}^{n}\sum_{k_3=1}^{n}\delta\left(k_1\frac{\overline{\tau}_1}{n},k_2\frac{\overline{\tau}_2}{n},k_3\frac{\overline{\tau}_3}{n}\right)
\varphi\left(t-k_1\frac{\overline{\tau}_1}{n},r_2-k_2\frac{\overline{\tau}_2}{n},r_3-k_3\frac{\overline{\tau}_3}{n}\right)+o(1).
\end{align*}
We now fix $\beta=3/2$ and set
\begin{align*}
\varphi_{\delta}(t,r_2,r_3,z_1,z_2,z_3)=\delta(z_1,z_2,z_3)\varphi(t-z_1,r_2-z_2,r_3-z_3).
\end{align*}
If we can show that
\begin{align}\label{Tn2}
&\sup_{s\in[s_0,1]}\sup_{t\in[0,\overline{\tau}_1],\atop r_2\in[0,\overline{\tau}_2],r_3\in[0,\overline{\tau}_3]}
\left|
\frac{\overline{\tau}_1\overline{\tau}_2\overline{\tau}_3}{n^3}
\sum_{k_1=1}^{\lfloor ns\rfloor}\sum_{k_2=1}^{n}\sum_{k_3=1}^{n}\varphi_{\delta}\left(t,r_2,r_3,k_1\frac{\overline{\tau}_1}{n},k_2\frac{\overline{\tau}_2}{n},k_3\frac{\overline{\tau}_3}{n}\right)\right.\notag\\&\left.
\hspace{4.1cm}-
\int_0^{s\overline{\tau}_1}\int_0^{\overline{\tau}_2}\int_0^{\overline{\tau}_3}
\!\varphi_{\delta}\left(t,r_2,r_3,z_1,z_2,z_3\right)\,\mathrm{d}z_3\mathrm{d}z_2\mathrm{d}z_1\right|
\end{align}
tends to zero as $n\to\infty$ the assertion follows, since uniform convergence always implies convergence in the Skorohod topology. 

If $\delta(t,r_2,r_3)$ is continuous we can proceed in an analogous way as in the proof of Theorem~3 of \cite{PawlakSteland}. Thus, it remains to treat the case that
$\delta(t,r_2,r_3)$ is of bounded variation in the sense of Hardy and Krause. Our aim is to apply the Hwlaka-Koksma inequality of Lemma~\ref{HKinequality}. As this inequality is formulated for integrals over the unit cube we first observe that
\begin{align*}
&\int_0^{s\overline{\tau}_1}\int_0^{\overline{\tau}_2}\int_0^{\overline{\tau}_2}
\!\varphi_{\delta}\left(t,r_2,r_3,z_1,z_2,z_3\right)\,\mathrm{d}z_3\mathrm{d}z_2\mathrm{d}z_1\\&
\qquad=
\int_0^{1}\int_0^{1}\int_0^{1}
\!s\overline{\tau}_1\overline{\tau}_2\overline{\tau}_3\varphi_{\delta}\left(t,r_2,r_3,s\overline{\tau}_1z_1,\overline{\tau}_2z_2,\overline{\tau}_3z_3\right)\,\mathrm{d}z_3\mathrm{d}z_2\mathrm{d}z_1.
\end{align*}
Put
$
\left(x_{k_1},x_{k_2},x_{k_3}\right)=\left(\frac{k_1\overline{\tau}_1}{n},\frac{k_2\overline{\tau}_2}{n},\frac{k_3\overline{\tau}_3}{n}\right)
$
and
$
\left(\widetilde{x}_{k_1},\widetilde{x}_{k_2},\widetilde{x}_{k_3}\right)=\left(\frac{k_1}{ns},\frac{k_2}{n},\frac{k_3}{n}\right),
$
for $\left(k_1,k_2,k_3\right)\in\{1,\ldots,\lfloor ns\rfloor\}\times\{1,\ldots,n\}\times\{1,\ldots,n\}$ and $s\geq s_0>0$. Then we obtain
\begin{align*}
&\frac{\overline{\tau}_1\overline{\tau}_2\overline{\tau}_3}{n^3}
\sum_{k_1=1}^{\lfloor ns\rfloor}\sum_{k_2=1}^{n}\sum_{k_3=1}^{n}\varphi_{\delta}\left(t,r_2,r_3,k_1\frac{\overline{\tau}_1}{n},k_2\frac{\overline{\tau}_2}{n},k_3\frac{\overline{\tau}_3}{n}\right)\\&
\qquad=
\frac{1}{sn^3}
\sum_{k_1=1}^{\lfloor ns\rfloor}\sum_{k_2=1}^{n}\sum_{k_3=1}^{n}s\overline{\tau}_1\overline{\tau}_2\overline{\tau}_3\varphi_{\delta}\left(t,r_2,r_3,s\overline{\tau}_1\widetilde{x}_{k_1},\overline{\tau}_2\widetilde{x}_{k_2},\overline{\tau}_3\widetilde{x}_{k_3}\right).
\end{align*}
Thus we can reformulate (\ref{Tn2}) as
\begin{align}\label{Tn2rewrite}
&\sup_{s\in[s_0,1]}\sup_{t\in[0,\overline{\tau}_1],\atop r_2\in[0,\overline{\tau}_2],r_3\in[0,\overline{\tau}_3]}
\left|
\frac{1}{sn^3}
\sum_{k_1=1}^{\lfloor ns\rfloor}\sum_{k_2=1}^{n}\sum_{k_3=1}^{n}s\overline{\tau}_1\overline{\tau}_2\overline{\tau}_3\varphi_{\delta}\left(t,r_2,r_3,s\overline{\tau}_1\widetilde{x}_{k_1},\overline{\tau}_2\widetilde{x}_{k_2},\overline{\tau}_3\widetilde{x}_{k_3}\right)\right.\\&\notag\left.
\hspace{4cm}-
\int_0^{1}\int_0^{1}\int_0^{1}
\!s\overline{\tau}_1\overline{\tau}_2\overline{\tau}_3\varphi_{\delta}\left(t,r_2,r_3,s\overline{\tau}_1z_1,\overline{\tau}_2z_2,\overline{\tau}_3z_3\right)\,\mathrm{d}z_3\mathrm{d}z_2\mathrm{d}z_1\right|.
\end{align}
An upper bound for this expression without the suprema is
\begin{align*}
&\quad\ \left|
\frac{1}{\lfloor ns\rfloor n^2}
\sum_{k_1=1}^{\lfloor ns\rfloor}\sum_{k_2=1}^{n}\sum_{k_3=1}^{n}s\overline{\tau}_1\overline{\tau}_2\overline{\tau}_3\varphi_{\delta}\left(t,r_2,r_3,s\overline{\tau}_1\widetilde{x}_{k_1},\overline{\tau}_2\widetilde{x}_{k_2},\overline{\tau}_3\widetilde{x}_{k_3}\right)\right.\\&\left.
-
\int_0^{1}\int_0^{1}\int_0^{1}
\!s\overline{\tau}_1\overline{\tau}_2\overline{\tau}_3\varphi_{\delta}\left(t,r_2,r_3,s\overline{\tau}_1z_1,\overline{\tau}_2z_2,\overline{\tau}_3z_3\right)\,\mathrm{d}z_3\mathrm{d}z_2\mathrm{d}z_1
\right|\\&
+
\left|
\left(\frac{1}{sn^3}-\frac{1}{\lfloor ns\rfloor n^2}\right)
\sum_{k_1=1}^{\lfloor ns\rfloor}\sum_{k_2=1}^{n}\sum_{k_3=1}^{n}s\overline{\tau}_1\overline{\tau}_2\overline{\tau}_3\varphi_{\delta}\left(t,r_2,r_3,s\overline{\tau}_1\widetilde{x}_{k_1},\overline{\tau}_2\widetilde{x}_{k_2},\overline{\tau}_3\widetilde{x}_{k_3}\right)
\right|\\&
\eqqcolon S_1+S_2.
\end{align*}
We first consider $S_2$. As both $\delta$ and $\varphi$ are of bounded variation in the sense of Hardy and Krause, $\varphi_{\delta}$ is bounded. Thus, for some constant $C\in\mathbb R$ we have
\[
S_2\leq C\lfloor ns\rfloor n^2\left|\frac{1}{\lfloor ns\rfloor n^2}-\frac{1}{sn^3}\right|
=C\left|1-\frac{\lfloor ns\rfloor}{ns}\right|
\leq C\left(1-\frac{ns-1}{ns}\right)
\leq \frac{C}{ns_0}
\]
which tends to zero as $n\to\infty$, uniformly for all $s\in[s_0,1]$ and $(t,r_2,r_3)\in[0,\overline{\tau}_1]\times[0,\overline{\tau}_2]\times[0,\overline{\tau}_3]$.

To estimate $S_1$ we apply the Hwlaka-Koksma inequality. If we put $K\coloneqq\{1,\ldots,\left\lfloor ns\right\rfloor\}\times\{1,\ldots,n\}\times\{1,\ldots,n\}$ and
\begin{align*}
D^{\star}_{N_s}\left(\widetilde{x}_{k_1},\widetilde{x}_{k_2},\widetilde{x}_{k_3}\right)\coloneqq
D^{\star}_{N_s}\left(\left\{\left(\widetilde{x}_{k_1},\widetilde{x}_{k_2},\widetilde{x}_{k_3}\right):(k_1,k_2,k_3)\in K\right\}\right),
\end{align*}
where $N_s\coloneqq\lfloor ns\rfloor n^2$, we have
\begin{align}\label{HKDisk}
S_1\leq V_{HK}\left(s\overline{\tau}_1\overline{\tau}_2\overline{\tau}_3\varphi_{\delta}\left(t,r_2,r_3,s\overline{\tau}_1\cdot,\overline{\tau}_2\cdot,\overline{\tau}_3\cdot\right),[0,1]^3\right)D^{\star}_{N_s}\left(\widetilde{x}_{k_1},\widetilde{x}_{k_2},\widetilde{x}_{k_3}\right).
\end{align}
By Proposition~11 in \cite{Owen} we can estimate the variation by
\begin{align*}
&\quad\ V_{HK}\left(s\overline{\tau}_1\overline{\tau}_2\overline{\tau}_3\varphi_{\delta}\left(t,r_2,r_3,s\overline{\tau}_1\cdot,\overline{\tau}_2\cdot,\overline{\tau}_3\cdot\right),[0,1]^3\right)\\&
\leq\overline{\tau}_1\overline{\tau}_2\overline{\tau}_3V_{HK}\left(\varphi_{\delta}\left(t,r_2,r_3,\cdot,\cdot,\cdot\right),[0,\overline{\tau}_1]\times[0,\overline{\tau}_2]\times[0,\overline{\tau}_3]\right).
\end{align*}
Since $\delta(t,r_2,r_3)$, by assumption, and $(z_1,z_2,z_3)\mapsto\varphi_{\delta}\left(t,r_2,r_3,s\overline{\tau}_1z_1,\overline{\tau}_2z_2,\overline{\tau}_3z_3\right)$, by a similar argument as in Lemma~\ref{phi_wBVHK}, are of bounded variation in the sense of Hardy and Krause uniformly in $t,r_2,r_3$, we obtain
\begin{align}\label{vhkphidelta}
\sup_{s\in[s_0,1]}\sup_{t\in[0,\overline{\tau}_1],\atop r_2\in[0,\overline{\tau}_2],r_3\in[0,\overline{\tau}_3]}V_{HK}\left(s\overline{\tau}_1\overline{\tau}_2\overline{\tau}_3\varphi_{\delta}\left(t,r_2,r_3,s\overline{\tau}_1\cdot,\overline{\tau}_2\cdot,\overline{\tau}_3\cdot\right),[0,1]^3\right)<\infty.
\end{align}
It remains to verify that the discrepancy $D^{\star}_{N_s}\left(\widetilde{x}_{k_1},\widetilde{x}_{k_2},\widetilde{x}_{k_3}\right)$ is $o(1)$. As for arbitrary $u_1,u_2,u_3\in[0,1)$ we have $\lfloor nsu_1\rfloor$ points $\widetilde{x}_{k_1}$ with $\widetilde{x}_{k_1}<u_1$, $\lfloor nu_2\rfloor$ points $\widetilde{x}_{k_2}$ with $\widetilde{x}_{k_2}<u_2$, and $\lfloor nu_3\rfloor$ points $\widetilde{x}_{k_3}$ with $\widetilde{x}_{k_3}<u_3$ we obtain
\begin{align*}
D^{\star}_{N_s}&\left(\widetilde{x}_{k_1},\widetilde{x}_{k_2},\widetilde{x}_{k_3}\right)
=
\sup_{0\leq u_1,u_2,u_3<1}\left|\frac{\lfloor nsu_1\rfloor\lfloor nu_2\rfloor\lfloor nu_3\rfloor}{\lfloor ns\rfloor n^2}-u_1u_2u_3\right|\\&
=
\sup_{0\leq u_1,u_2,u_3<1}\left|\frac{(nsu_1-\varepsilon_1)(nu_2-\varepsilon_2)(nu_3-\varepsilon_3)-(ns-\varepsilon)n^2u_1u_2u_3}{(ns-\varepsilon)n^2}\right|
\end{align*}
for appropiately chosen $\varepsilon,\varepsilon_1,\varepsilon_2,\varepsilon_3\in[0,1)$. This finally leads to
\begin{align}\label{diskpart}
D^{\star}_{N_s}\left(\widetilde{x}_{k_1},\widetilde{x}_{k_2},\widetilde{x}_{k_3}\right)
\leq\frac{4n^2+3n+1}{(ns-1)n^2}
\leq\frac{4n^2+3n+1}{(ns_0-1)n^2}=O\left(\frac{1}{n}\right),
\end{align}
uniformly for all $s\in[s_0,1]$.
Now, combining (\ref{vhkphidelta}) and (\ref{diskpart}) with (\ref{HKDisk}) it follows that $S_1\to0$ as $n\to\infty$, uniformly in $s,t,r_2,r_3$. Thus, assertion (\ref{Tn2}) also follows for the case that $\delta$ is a function of bounded variation in the sense of Hardy and Krause which finally completes the proof.
\end{proof}
\begin{proof}[Proof of Lemma \ref{cltdetalt}]
The assertion directly follows with Theorem~\ref{detalt} and the continuous mapping theorem.
\end{proof}
\begin{proof}[Proof of Theorem \ref{randalt}]
The proof is analogoulsy to the one of Theorem~\ref{detalt}.
\end{proof}
\begin{proof}[Proof of Theorem \ref{basicresultmvbw2}]
As $w$ fulfills the same assumptions as $\varphi$ one can adopt the proof of Theorem~\ref{FCLTNull}.
\end{proof}
\begin{proof}[Proof of Theorem \ref{unknownrefsignal}]
The result easily follows from the fact that under Assumption~1 $ y_{i,j,k}-\overline{y}_{\cdot,j,k} = \varepsilon_{i,j,k}-\overline{\varepsilon}_{\cdot,j,k} $ and thus
\begin{align*}
&\quad\ \ \frac{1}{\sqrt{n_1n_2n_3}}\sum_{i=1}^{\left\lfloor n_1v_1\right\rfloor}\sum_{j=1}^{\left\lfloor n_2v_2\right\rfloor}\sum_{k=1}^{\left\lfloor n_3v_3\right\rfloor}(y_{i,j,k}-\overline{y}_{\cdot,j,k})\\&
%
%
=\frac{1}{\sqrt{n_1n_2n_3}}\sum_{i=1}^{\left\lfloor n_1v_1\right\rfloor}\sum_{j=1}^{\left\lfloor n_2v_2\right\rfloor}\sum_{k=1}^{\left\lfloor n_3v_3\right\rfloor}\varepsilon_{i,j,k}
-\frac{1}{\sqrt{n_1n_2n_3}}\frac{\left\lfloor n_1v_1\right\rfloor}{\left\lfloor n_1s_0\right\rfloor}\sum_{l=1}^{\left\lfloor n_1s_0\right\rfloor}\sum_{j=1}^{\left\lfloor n_2v_2\right\rfloor}\sum_{k=1}^{\left\lfloor n_3v_3\right\rfloor}\varepsilon_{l,j,k}\\&
\Rightarrow\sigma\left(B(v_1,v_2,v_3)-\frac{v_1}{s_0}B(s_0,v_2,v_3)\right)\eqqcolon\sigma B^c(v_1,v_2,v_3),
\end{align*}
as $\min_{1\leq i\leq 3}n_i\to\infty$. Now one may argue as in the proof of Theorem \ref{FCLTNull}.
\end{proof}


\end{document}